\documentclass[12pt,reqno]{amsart}

\usepackage[utf8]{inputenc}
\usepackage{graphicx}
\usepackage{amsmath,amssymb}
\usepackage{hyperref}
\usepackage{xcolor}
\usepackage{bm}
\usepackage{comment}
\usepackage{cite}
\usepackage{times}

\renewcommand{\div}{\rm div}

\numberwithin{equation}{section}

\allowdisplaybreaks

\newtheorem{theorem}{Theorem}[section]
\newtheorem*{theorem*}{Theorem}
\newtheorem{definition}[theorem]{Definition}

\newtheorem{remark}[theorem]{Remark}
\newtheorem{lemma}[theorem]{Lemma}
\newtheorem{proposition}[theorem]{Proposition}
\newtheorem{corollary}[theorem]{Corollary}

\def\O{\Omega}

\def\P{\mathcal{P}}

\def\dist{{\rm dist}} 
\def\diam{{\rm diam}}
\def\div{{\rm div}\,}
\def\supp{{\rm supp}}
\def\dx{{\,\rm d}x}
\def\dy{{\,\rm d}y}
\def\dz{{\,\rm d}z}

\newcommand{\R}{{\mathbb R}}
\newcommand{\N}{{\mathbb N}}

\DeclareMathOperator*{\esssup}{ess\,sup}
\DeclareMathOperator*{\essinf}{ess\,inf}

\newcommand{\pp}{{p(\cdot)}}
\newcommand{\cpp}{{p'(\cdot)}}
\newcommand{\Lp}{L^{p(\cdot)}}
\newcommand{\Lps}{L^{p^*(\cdot)}}
\newcommand{\Pp}{\mathcal P}

\newcommand{\qq}{{q(\cdot)}}

\newcommand{\cqq}{{q'(\cdot)}}

\newcommand{\ps}{{p^*(\cdot)}}

\newcommand{\grad}{\nabla}

\newcommand{\tauLH}{\partial LH_0^\tau}

\title[Poincar\'e and Sobolev inequalities and log-Holder continuity at the boundary]{Poincar\'e and Sobolev inequalities with variable exponents and  log-Holder continuity only at the boundary}

\date{\today}

\author[D. Cruz-Uribe]{David Cruz-Uribe OFS}
\address{Department of Mathematics\\ University of Alabama\\ Tuscaloosa, AL 35487, US} 
\email{dcruzuribe@ua.edu}

\author[F. L\'opez-Garc\'\i a]{Fernando L\'opez-Garc\'\i a}
\address{Department of Mathematics and Statistics\\ California State Polytechnic University Pomona\\ Pomona, CA 91768, US}
\email{fal@cpp.edu}

\author[I. Ojea]{Ignacio Ojea}
\address{Departamento de Matemática \\ Facultad de Ciencias Exactas y Naturales\\ Universidad de Buenos Aires \\ 
IMAS - CONICET} 
\email{iojea@dm.uba.ar}

\begin{document}

\begin{abstract}
We prove Sobolev-Poincar\'e and Poincar\'e inequalities in variable Lebesgue spaces $\Lp(\Omega)$, with $\Omega\subset\R^n$ a bounded John domain, with weaker regularity assumptions on the exponent $\pp$ that have been used previously. In particular, we require $\pp$ to satisfy a new \emph{boundary $\log$-H\"older condition} that imposes some logarithmic decay on the oscillation of $\pp$ towards the boundary of the domain. Some control over the interior oscillation of $\pp$ is also needed, but it is given by a very general condition that allows $\pp$ to be discontinuous at every point of $\Omega$. Our results follows from a local-to-global argument based on the continuity of certain Hardy type operators. We provide examples that show that our boundary $\log$-H\"older condition is essentially necessary for our main results. The same examples are adapted to show that this condition is not sufficient for other related inequalities. Finally, we give an application to a Neumann problem for a degenerate $\pp$-Laplacian.   
\end{abstract}



\thanks{The first author is partially supported by a Simons Foundation
  Travel Support for Mathematicians Grant and by NSF Grant DMS-2349550.  
  The third author is partially supported by ANPCyT under grand PICT 2018-3017, by CONICET under grant PIP112201130100184CO and by Universidad de Buenos Aires under grant 20020170100056BA.}

\maketitle

\section{Introduction}
In this paper we prove Poincar\'e and Sobolev inequalities in the variable Lebesgue spaces $\Lp(\Omega)$, with weaker regularity assumptions on the exponent function $\pp$ than have been used previously.  To put our results in context, we briefly survey some earlier results in the classical Lebesgue spaces. For a more detailed history of these results, with proofs, see~\cite{Hajlasz2001}.
A Sobolev-Poincaré inequality is an inequality of the form
\begin{equation} \label{eqn:old-ps-inequality}
\|f-f_\O\|_{L^q(\O)} \le C\|\nabla f\|_{L^p(\O)},
\end{equation}
where $\O\subset \R^n$ is a domain, $f$ is locally Lipschitz on $\O$, and $f_\O=\frac{1}{|\O|}\int_\O f\,dx$.  Some assumptions need to be made on the domain $\O$, and on the exponents $p$ and $q$, for this inequality to hold. For $1\le p<n$, define $    p^*=\frac{np}{n-p}$. If $\Omega$ is a Lipschitz domain and $1<p<n$, \eqref{eqn:old-ps-inequality} was proved for $q=p^*$ by Sobolev~\cite{Sobolev1,Sobolev2}, and extended to the case $p=1$ and $q=1^*$ by Gagliardo~\cite{Gagliardo} and Nirenberg~\cite{Nirenberg}. The inequality for $q<p^*$ follows from the critical case $q=p^*$ by  H\"older's inequality. Inequality~\eqref{eqn:old-ps-inequality} was extended to bounded John domains by Martio \cite{Martio} when $1<p<n$ and by Bojarski \cite{MR0982072} when $p=1$. Moreover, Buckley and Koskela~\cite{BuckleyKoskela} showed that John domains are the largest class of bounded domains for which the Sobolev-Poincaré inequality holds with  $p<n$ and $q=p^*$. 
If $p<n$, then as $p\rightarrow n$, $p^*\rightarrow \infty$, but the constant on the right-hand side of~\eqref{eqn:old-ps-inequality} blows up. 
Inequality ~\eqref{eqn:old-ps-inequality} with $p=n$ does not hold for $q=\infty$, but for $p\ge n$ it holds  for every $1<q<\infty$. This case follows from the Sobolev embedding theorem; see~\cite[Corollary~4.2.3]{MR1014685}.  

Inequality~\eqref{eqn:old-ps-inequality} is a particular case of a  \emph{weighted} inequality of the form
\begin{equation}\label{imp-poincare}
\|f - f_\Omega\|_{L^q(\O)} \le C\|d^{1-\alpha} \nabla f\|_{L^p(\O)},
\end{equation}
where $d(x) = $dist$(x,\partial\O)$, $\alpha\in[0,1]$, $p\alpha<n$, and $1<p\le q\le \frac{np}{n-p\alpha}$. Inequality~\eqref{imp-poincare} was first proved on bounded John domains by Hurri-Syrj\"anen~\cite{H2}, and later extended with additional weights by Drelichman and Dur\'an~\cite{DD1}. An important special case of this inequality is when $q=p$ and $\alpha=1$; this case is usually referred to as an~\emph{improved Poincaré inequality}.

A Sobolev inequality is an inequality of the form
\[ \|f\|_{L^q(\O)} \le C\|\nabla f\|_{L^p(\O)},\]
where $f$ is assumed to be a Lipschitz function with compact support contained in $\Omega$.  By H\"older's inequality and the triangle inequality, the Poincar\'e inequality~\eqref{eqn:old-ps-inequality} implies the Sobolev inequality.  However, one key difference is that since this inequality is for functions of compact support, their behavior near the boundary, and so the geometry of the set $\Omega$, does not play an important role.  

\medskip

Sobolev and Sobolev-Poincaré inequalities have also been proved on the variable Lebesgue spaces. Intuitively, given an exponent function $\pp : \Omega \rightarrow [1,\infty)$, the space $\Lp(\Omega)$ consists of all measurable functions $f$ such that
\[ \int_\Omega |f(x)|^{p(x)}\,dx < \infty.  \]
For a precise definition and further details, see Section~\ref{section:prelim} below and~\cite{CruzUribeFiorenza}.  
Harjulehto and H\"ast\"o proved that if $\Omega$ is a John domain, and the exponent function $\pp$ is such that either $p_-(\Omega)<n$ and $p_+(\Omega) \leq p_-(\Omega)^*$, or if $p_-(\Omega)\geq n$ and $p_+(\Omega)<\infty$, then
\begin{equation} \label{eqn:p-p-Poincare}
\|f-f_\Omega\|_{L^\pp(\Omega)} \leq C\|\grad f\|_{L^\pp(\Omega)}.  
\end{equation}
As a consequence of this result they showed that if $\Omega$ is bounded, convex domain and $\pp$ is uniformly continuous on $\Omega$, then~\eqref{eqn:p-p-Poincare} holds.  

In \cite[Lemma 8.3.3]{DieningBook} Diening, {\em et al.} proved that given a John domain $\Omega$, if  the exponent $\pp$ satisfies $1<p_-(\Omega) \le p_+(\Omega) <n$ and is log-H\"older continuous, that is,
\[ |p(x)-p(y)| \leq \frac{C_0}{-\log(|x-y|)}, \qquad  x,\,y \in \Omega, \; |x-y|<\tfrac{1}{2}, \]
and if $\qq = p^*(\cdot)$, then for all locally Lipschitz functions $f$,
\begin{equation} \label{eqn:var-poincare}
 \|f-f_\Omega\|_{L^\qq(\Omega)} \leq C\|\grad f\|_{L^\pp(\Omega)}.  
 \end{equation}
This result has been generalized to H\"ormander vector fields by Li, Lu, and Tang~\cite{MR3360773}.

Sobolev inequalities have also been considered by a number of authors.  Kov\'a\v cik and R\'akosn\'ik~\cite{MR1134951} proved that
\begin{equation} \label{eqn:var-sobolev-old}
\|f\|_{L^\qq(\Omega)} \leq C\|\grad f\|_{L^\pp(\Omega)}
\end{equation}
where $\pp$ is uniformly continuous on $\overline{\Omega}$, $p_+<n$,  and $q(x)=p^*(x)-\varepsilon$ for some $\varepsilon>0$.  Edmunds and 
R\'akosn\'ik~\cite{MR1815935,MR1944549} proved that~\eqref{eqn:var-sobolev-old} holds with $\qq=\ps$, provided that $\Omega$ has Lipschitz boundary and $\pp$ is Lipschitz, later extending this to $\pp$ is H\"older continuous.
In \cite[Theorem 6.29]{CruzUribeFiorenza} the first author and Fiorenza proved that~\eqref{eqn:var-sobolev-old} holds with $\qq=\ps$ provided that $\pp$ is log-H\"older continuous on $\Omega$. (Though not stated there, their argument also shows that \eqref{eqn:old-ps-inequality} holds with $\qq=\ps$ if $\Omega$ is convex.)   A different proof was given by Diening, {\em et al.}~\cite[Theorem~8.3.1]{DieningBook} It is worth noting that in \cite{CruzUribeFiorenza} the hypotheses on $\pp$ are actually somewhat weaker than 
log-H\"older continuity:  they instead assume that the Hardy-Littlewood maximal operator is bounded on
$L^{(p^*(\cdot)/n')'}(\O)$. This hypothesis is a consequence of the proof, which uses the theory of Rubio de Francia extrapolation on variable Lebesgue spaces.    For a discussion of how this hypothesis differs from assuming log-H\"older continuity, see~\cite[Chapter~4]{CruzUribeFiorenza}.  A very different version of Sobolev's inequality was proved by Mercaldo, {\em et al.}~\cite[Proposition~2.4, Remark~1.2]{MR2906546}.  They showed that if $\pp$ took on two distinct values $1<p_1<p_2\leq 2$, and the sets where it took on these values have Lipschitz boundaries, then~\ref{eqn:var-sobolev-old} holds with $\qq=\pp$. 

Variable Lebesgue space versions of the weighted Poincar\'e inequality~\ref{imp-poincare} have not been explicitly proved in the literature.  However, arguing as in the proof of~\cite[Theorem~6.29]{CruzUribeFiorenza}, it is possible to prove the following result. Let $\Omega$ be a bounded John domain and fix $0\leq \alpha <1$.  If the exponent $\pp$ satisfies $1<p_-(\Omega) \le p_+(\Omega) <n/\alpha$ and is log-H\"older continuous, and $\qq \leq \frac{n\pp}{n-\pp\alpha}$, then
\begin{equation} \label{eqn:var-imp-poincare}
\|f - f_\Omega\|_{L^\qq(\Omega)} \le C\|d^{1-\alpha} \nabla f\|_{L^\pp(\O)}.
\end{equation}
The proof uses Rubio de Francia extrapolation, starting from the weighted versions of this inequality proved in~\cite[Theorems~3.3,~3.4]{DD1}.  Details are left to the reader.  We note that, instead of assuming log-H\"older continuity, it is also possible to state the hypotheses of this result in terms of the boundedness of the maximal operator.  

\bigskip

In all of the above results, it was necessary to assume some additional continuity or control on the oscillation of the exponent function $\pp$.  For the original Poincar\'e inequality proved by Harjulehto and H\"ast\"o they either required a strong restriction on the global oscillation of the exponent function, or uniform continuity.  Similarly, the Sobolev inequality of Mercaldo, {\em et al.} required a very particular form for the exponent function.  In the later results for Sobolev and Sobolev-Poincar\'e inequalities, the exponent was required to satisfy $p_+<n$ and to be log-H\"older continuous or the Hardy-Littlewood maximal operator is  bounded on a particular space that arises in the underlying extrapolation argument, a condition which, particularly in applications, is very close to log-H\"older continuity.  In~\cite[Problems~A.21, A.22]{CruzUribeFiorenza} the authors asked about weaker regularity assumptions for proving Poincar\'e and Sobolev inequalities.

In this paper we show that we can considerably weaken the previous regularity conditions and still prove the weighted Poincar\'e inequality~\eqref{eqn:var-imp-poincare}.    We will require three conditions.  First, we replace log-H\"older continuity with a condition that essentially implies that $\pp$ is log-H\"older continuous at the boundary.  Given  $x\in \Omega$ and $\tau\geq 1$ we define $d(x)=d(x,\partial\Omega)$ and $B_{x,\tau} = B(x,\tau d(x))$.

\begin{definition}\label{tauLHB}
Given $\Omega\subset\R^n$, let $\pp \in \Pp(\Omega)$ and $\tau\ge 1$. We say that $\pp$ satisfies the $\tauLH$ condition  if there is a constant $C_0$, that may depend on $\tau$, such that: 
\begin{equation}\label{Boundary LH constant}
p_+(B_{x,\tau}) - p_-(B_{x,\tau}) \le \frac{C_0}{-\log(\tau d(x))},
\end{equation}
for all $x\in\Omega$ with $\tau d(x)\leq 1/2$. In that case, we write $\pp\in \tauLH(\Omega)$.
\end{definition}

Below, we will show that with modest additional conditions on $\partial \Omega$, an exponent $\pp \in \tauLH(\Omega)$ can be extended to a function that is log-H\"older continuous on $\partial \Omega$.  See Section~\ref{section:lh-boundary}.

Second, we require a weak continuity property in the interior of $\Omega$.

\begin{definition}\label{def:eps-cont}
Given a domain $\O\subset\R^n$, $\varepsilon>0$,  and $f:\O\to\R$, we say that the function $f$ is $\varepsilon$-continuous at a point $x\in\O$ if there exists $\delta>0$ such that 
$|f(y)-f(x)|<\varepsilon$ for every $y \in B(x,\delta)$. 
If  $f$ is $\varepsilon$-continuous at every $x\in \O$, we say that $f$ is $\varepsilon$-continuous on  $\O$.
Finally, we say that $f$ is \emph{uniformly} $\varepsilon$-continuous if the same $\delta>0$ can be taken for every $x\in\Omega$.
\end{definition}

\begin{remark}
    The notion of $\varepsilon$-continuity given in Definition \ref{def:eps-cont} is well-known: it arises, for example, in the proof that a bounded function is Riemann integrable if and only if its set of discontinuities has Lebesgue measure $0$.  See Apostol~\cite[Section~9-21]{MR0087718}, or Convertito and the first author~\cite[Section~2.1]{MR4701107}.
\end{remark}
It is immediate that if $f$ is continuous on $\Omega$, then it is $\varepsilon$-continuous for every $\varepsilon>0$.  However, $\varepsilon$-continuous functions may be discontinuous:  for example, any step function on $\R$ whose discontinuities have a jump smaller than $\varepsilon$.  In particular, there exist $\varepsilon$-continuous exponent functions $\pp$ for which the maximal operator is unbounded on $\Lp(\Omega)$:  see~\cite[Example~3.21]{CruzUribeFiorenza}.  If $f$ is $\varepsilon$-continuous on $\overline{\Omega}$, then by a standard compactness argument, $f$ is uniformly $2\varepsilon$-continuous on $\Omega$.  We leave the details to the reader.

Finally, we need to restrict the relationship between $\pp$ and $\qq$ in a natural way.   We will assume that $\pp\in \Pp(\Omega)$ satisfies $1<p_-(\O)\le p_+(\O)<\infty$, and that $\qq\in \Pp(\Omega)$ is defined by
\begin{equation}\label{eq:def q-intro}
\frac{1}{\pp}-  \frac{\alpha}{n} = \frac{1}{\qq}
\end{equation}
where $\alpha$ satisfies
\begin{equation}\label{eq:alpha-intro}
\begin{cases} 0\le \alpha < 1 &\textrm{ if } p_+(\O)<n \\
0\le \alpha< \frac{n}{p_+(\O)} & \textrm{ if } p_+(\O)\geq n. 
\end{cases}
\end{equation}
Note that with this definition  our results include the case $p_+>n$.

\begin{theorem}\label{theorem:sp-general-intro}
Let $\O\subset \R^n$ be a bounded John domain.  Suppose  $\pp\in\Pp(\Omega)$ is such that $1<p_-(\O)\le p_+(\O)<\infty$ and $\pp\in\partial LH_0^{\tau_K}(\Omega)$, where the constant $\tau_K\geq 1$ depends on the John domain constants of $\Omega$.  Fix $\alpha$ as in~\eqref{eq:alpha-intro} and define $\qq$ by~\eqref{eq:def q-intro}.  Suppose also that $\frac{1}{\pp}$ is uniformly $\frac{\sigma}{n}$-continuous for some $\sigma<1-\alpha$.  Then there is a constant $C$ such that for every $f\in W^{1,p(\cdot)}(\Omega)$,
\begin{equation}\label{eq:sp-john-intro}
\|f(x)-f_\Omega\|_{L^\qq(\Omega)} \le C \|d^{1-\alpha}\nabla f(x)\|_{L^\pp(\Omega)},
\end{equation}
where $d(x)=\dist(x,\partial \Omega)$.
\end{theorem}

As a consequence of Theorem~\ref{theorem:sp-general-intro} we get a Sobolev inequality with the same exponents but without assuming log-H\"older continuity at the boundary.

\begin{theorem} \label{thm:sobolev-intro}
Let $\O\subset \R^n$ be a bounded  domain.  Suppose  $\pp\in\Pp(\Omega)$ is such that $1<p_-(\O)\le p_+(\O)<\infty$.  Fix $\alpha$ as in~\eqref{eq:alpha-intro} and define $\qq$ by~\eqref{eq:def q-intro}.  Suppose also that $\frac{1}{\pp}$ is uniformly $\frac{\sigma}{n}$-continuous for some $\sigma<1-\alpha$.  Then there is a constant $C$ such that for every $f\in W_0^{1,p(\cdot)}(\Omega)$,
\begin{equation*} 
    \|f\|_{L^\qq(\Omega)} \le C \|\nabla f(x)\|_{L^\pp(\Omega)}.
\end{equation*}
\end{theorem}

One seeming drawback to Theorems~\ref{theorem:sp-general-intro} and~\ref{thm:sobolev-intro} is that we are not able to prove the inequality for $\qq=\ps$, but only for $\qq<\ps$.
We will discuss the technical reason for the restriction that $0\le \alpha<1$ in \eqref{eq:alpha-intro} below:  see Remark~\ref{rmk:alpha=1}. 
Broadly speaking, however, the problems lies in the \emph{local} inequality on cubes. Our arguments in the interior of $\Omega$ are  based on the constant exponent Sobolev-Poincaré inequality, and so require only weak regularity requirements on  the exponents.  However, this approach has a gap that prevents us from reaching the critical case. This gap, however, cannot be easily closed.  We construct an example of uniformly continuous exponent $\pp$ that is not log-H\"older continuous such that the Sobolev and Sobolev-Poincar\'e inequalities fail to hold with $\qq=\ps$:  see Section~\ref{section:necessity}.  This gives a positive answer to \cite[Problem~A.22]{CruzUribeFiorenza} and raises the question of determining the  domains $\Omega$ and continuity conditions stronger than uniform continuity but weaker than log-H\"older continuity for which these inequalities hold.

We also note that while Theorem~\ref{thm:sobolev-intro} applies to much more general domains and exponent functions $\pp$ than the result of Mercaldo, {\em et al.}, we are not actually able to recapture their result.  In their result their exponent $\pp$ is $\frac{1}{2}$-continuous, and (as we will see below in Corollary~\ref{cor:imp-poincare}) we need to assume something stronger than $\frac{1}{n}$-continuity.   It is an interesting question as to whether and to what extent their results on the $\pp$-Laplacian can be generalized using our results. 

The basic idea in our proof of Theorem~\ref{theorem:sp-general-intro} is to cover the domain $\Omega$ with cubes from a Whitney decomposition and prove a local Sobolev-Poincar\'e inequality on each cube.  Different techniques are required for cubes in the interior and those close to the boundary.  We then use the fact that $\Omega$ is a John domain to apply a local-to-global argument to patch together the local estimates.  We prove Theorem~\ref{thm:sobolev-intro} using the well-known fact that a Sobolev-Poincar\'e implies a Sobolev inequality, provided that one can estimate the average term $f_\Omega$ appropriately.  The main tool we need for this argument is an extension theorem for $\varepsilon$-continuous exponent functions that allows us to avoid assuming that the domain is a John domain and that we have the log-H\"older condition at the boundary.  It would be interesting to give a direct proof that did not pass through the Sobolev-Poincar\'e inequality.

The remainder of this paper is organized as follows.  In Section~\ref{section:prelim} we gather some preliminary results about variable Lebesgue and Sobolev spaces, and prove some basic properties of the $\tauLH(\Omega)$ condition.  In Section~\ref{section:hardy} we introduce a tree structure defined using the Whitney decomposition of a domain, and prove a boundedness result for a Hardy-type operator, $A_\Gamma$,  defined with respect to this tree structure.  In Section~\ref{section:var-Hardy} we introduce two more Hardy-type operators, $T_\pp$ and $T_\pp^\alpha$, where the averages over the Whitney cubes are formed using the variable Lebesgue space norm, and explore the role played by the $\tauLH(\Omega)$ condition.  We note that the operator $T_\pp$ is gotten by taking the parameter $\alpha=0$ in $T_\pp^\alpha$, but we have chosen to treat them separately, despite some repetition.  We have done so since a careful analysis of the difference between the proofs shows exactly where the restriction~\eqref{eq:alpha-intro} on $\alpha$ in Theorem~\ref{theorem:sp-general-intro} comes from.  In Section~\ref{section:john} we prove a decomposition theorem for functions on a John domain.  This decomposition is central to our ability to extend Sobolev-Poincar\'e inequalities defined on cubes to a John domain.  

In Section~\ref{section:improved-SPi} we prove Theorem~\ref{theorem:sp-general-intro}.  In order to do so we first prove a number of Sobolev-Poincar\'e inequalities on cubes.  Some of these results are known but we give all the details as we need to keep very careful track of the constants. We also give two corollaries that hold if we restrict  the global oscillation of $\pp$:  see Theorems~\ref{thm:sp-small-pp} and~\ref{thm:sp-big-pp}. In Section~\ref{section:sobolev} we prove Theorem~\ref{thm:sobolev-intro}; as we noted above the heart of the proof is an extension theorem for $\varepsilon$-continous functions; our proof is adapted from the proofs of~\cite[Chapter~ VI, Section~2]{S_SingularIntegrals} and~\cite[Lemma~2.4]{CruzUribeFiorenza}.   

In Section~\ref{section:necessity} we consider the question of the necessity of the boundary log-H\"older condition.  We cannot prove that it is necessary in general, but we show that it is very close to necessary: we construct a continuous exponent $\pp$ that is not in $\tauLH(\Omega)$, and inequality~\eqref{eq:sp-john-intro} fails to hold.  This is analogous to the situation for the boundedness of the Hardy-Littlewood maximal operator:  log-H\"older continuity is not necessary, but examples show that this is the weakest continuity condition which universally guarantees boundedness.  Our construction is very general, and so we adopt it to give additional examples.  First we construct a uniformly continuous exponent that is not log-H\"older continuous such that the Sobolev and Sobolev-Poincaré inequalities fail.  We then modify this example to show that the Korn inequality need not be true on $\Lp(\Omega)$ if we do not assume $\pp$ is log-H\"older continuous in the interior.  For this same example we show that the divergence equation cannot be solved in $\Lp(\Omega)$.  

  In Section~\ref{section:lh-boundary} we show that with modestly stronger assumptions on the domain, if an exponent function $\pp$ satisfies the $\tauLH(\Omega)$ condition, then it extends in a natural way to a function defined on $\partial \Omega$ that is log-H\"older continuous on the boundary.  This further reinforces referring to $\tauLH(\Omega)$ as boundary log-H\"older continuity.    Finally, in Section~\ref{section:applications} we apply our results to a problem in degenerate elliptic PDEs.  We use our improved Poincar\'e inequalities  and a result due to the first author,  Penrod, and Rodney~\cite{MR4332462},  to give solutions to a Neumann-type problem for a degenerate $\pp$-Laplacian.   

Throughout this paper, all notation is standard or will be defined as needed.  By $n$ we will always mean the dimension of the underlying space, $\R^n$.  Constants will be denoted by $C$, $c$, etc. and may change in value from line to line.  Given two quantities $A$ and $B$, if for some $c>0$ $A\leq cB$, then we will write $A\lesssim B$.  If $A\lesssim B$ and $B\lesssim A$, we will write $A\approx B$.

\section{Variable Lebesgue spaces and the boundary log-H\"older condition}
\label{section:prelim}

We begin by defining the variable Lebesgue space $\Lp(\Omega)$.  For complete information, see~\cite{CruzUribeFiorenza}.
Given a domain $\Omega\subset\R^n$, let $\Pp(\Omega)$ be the set of all Lebesgue measurable functions $p(\cdot):\Omega \to [1,\infty]$.  Given any measurable set $E\subset \Omega$, define
\[ p_-(E) = \essinf_{x\in E} p(x), \qquad p_+(E) = \esssup_{x\in E} p(x).  \]
For brevity we will write $p_-=p_-(\Omega)$ and $p_+ = p_+(\Omega)$.  Let $\Omega_\infty = \{ x\in \Omega : p(x)=\infty\}$ and $\Omega_0 = \Omega\setminus \Omega_\infty$. Define the space $\Lp(\Omega)$ to be the collection of all measurable functions $f$ such that for some $\lambda>0$,
\[ \rho_\pp(f/\lambda) = \int_{\Omega_0} \bigg(\frac{|f(x)|}{\lambda}\bigg)^{p(x)}\,dx + \lambda^{-1}\|f\|_{L^\infty(\Omega)} < \infty. \]
This becomes a Banach function space when equipped with the Luxemburg norm
\[ \|f\|_{\Lp(\Omega)} = \inf\{ \lambda > 0 : \rho_\pp(f/\lambda) \leq 1 \}.  \]
We will be working with functions in the variable Sobolev spaces.  Define  $W^{1,\pp}(\Omega)$ to be the space of  all functions $f\in W^{1,1}_{loc}(\Omega)$ (that is, locally integrable functions whose weak derivatives exist and are locally integrable) such that $f,\,\grad f \in \Lp(\Omega)$. Define $W^{1,\pp}_0(\Omega)$ to be the closure of Lipschitz functions of compact support in $W^{1,\pp}(\Omega)$ with respect to the norm $\|f\|_{W^{1,\pp}(\Omega)}=\|f\|_{L^\pp(\Omega)}+\|\grad f\|_{\Lp(\Omega)}$. 

Given $\pp \in \Pp(\Omega)$, define $\cpp\in \Pp(\Omega)$, the dual exponent function, pointwise by
\[ \frac{1}{p(x)} + \frac{1}{p'(x)} = 1,\]
with the convention that $1/\infty=0$.  We then have an equivalent expression for the norm, referred to as the associate norm:  
\begin{equation} \label{eqn:assoc-norm}
 \|f\|_{\Lp(\Omega)} \approx \sup_{\|g\|_{L^\cpp(\Omega)}\leq 1} \int_\Omega f(x)g(x)\,dx.
\end{equation}
We also have a version of H\"older's inequality:  given $f\in \Lp(\Omega)$ and $g \in L^\cpp(\Omega)$,
\begin{equation} \label{eqn:holder}
\int_\Omega |f(x)g(x)|\,dx \leq C\|f\|_{\Lp(\Omega)}\|g\|_{L^\cpp(\Omega)}. 
\end{equation}

If $\Omega$ is a bounded domain, then we have the following general embedding theorem.

\begin{lemma} \label{lemma:pp-qq-imbed}
Let $\Omega$ be a bounded domain.  Given $\pp,\,\qq \in \Pp(\Omega) $, if $p(x) \leq q(x)$ for almost every $x\in \Omega$, then $L^\qq(\Omega) \subset \Lp(\Omega)$ and $\|f\|_\pp \leq (1+|\Omega|)\|f\|_\qq$. 
\end{lemma}

\medskip

We now prove some properties of the boundary log-H\"older condition, Definition~\ref{tauLHB}. 
We first note that the bound $\tau d(x) \leq 1/2$ in this definition  is somewhat arbitrary:  the intention is that this condition holds for $x$ very close to $\partial \Omega$.  In particular, if $p_+<\infty$ and $0<a<1$, we get an equivalent condition if we assume $\tau d(x) \leq a$. In this case we also get that the classes are nested:  if $\tau_1<\tau_2$, then $\partial LH_0^{\tau_2}(\Omega) \subset \partial LH_0^{\tau_1}(\Omega)$.  To see this, we will show that~\eqref{Boundary LH constant} holds for $\tau_1$ when $\tau_1 d(x)\leq (\tau_1/\tau_2)^2<1$.  Then in this case, $\tau_2 d(x)\leq \tau_1/\tau_2<1$, so we may assume~\eqref{Boundary LH constant} holds for $\tau_2$.  But then we have that 
\begin{multline*}
     p_+(B_{x,\tau_1}) - p_-(B_{x,\tau_1}) 
\leq p_+(B_{x,\tau_2}) - p_-(B_{x,\tau_2}) \\
\leq \frac{C_{\tau_2}}{-\log(\tau_2 d(x))}
= \frac{C_{\tau_2}}{-\log(\tau_2/\tau_1) -\log(\tau_1 d(x))}
\leq \frac{2C_{\tau_2}}{-\log(\tau_1 d(x))}.
\end{multline*}

{  
\begin{lemma} \label{lemma:dual-LHB}
Given $\pp \in \tauLH(\Omega)$, if $p_-(B_{x,\tau}) \geq \gamma >1$ for every $x\in \Omega$, then $\cpp \in \tauLH(\Omega)$.
\end{lemma}

\begin{proof}
Given a set $E\subset \Omega$, we have that 
\[ p_+'(E) = \esssup_{x\in E} p'(x)= (p_-(E))', \qquad p_-'(E) = \essinf_{x\in E} p'(x) = (p_+(E))'.  \]
Therefore,
\begin{multline*}
    p_+'(B_{x,\tau}) - p_-'(B_{x,\tau})
    = (p_-(B_{x,\tau})' - (p_+(B_{x,\tau}))' \\
    = \frac{p_+(B_{x,\tau})- p_-(B_{x,\tau})}{(p_-(B_{x,\tau})-1)(p_+(B_{x,\tau})-1}
    \leq \gamma^{-2} (p_+(B_{x,\tau})- p_-(B_{x,\tau})) 
    \leq  \frac{\gamma^{-2}C_0}{-\log(\tau d(x))}.
\end{multline*}
\end{proof}
}

As we shall see, for proving our main results we will only need $\pp\in \tauLH(\Omega)$ for certain $\tau$ depending on the domain. However this particular value of $\tau$ can be unknown, so we give also the following more restrictive definition: 

\begin{definition}\label{LHB}
We say that $\pp$ is uniformly log-H\"older continuous at the boundary, and write $\pp\in \partial LH_0(\Omega)$, if $\pp\in \tauLH(\Omega)$ for every $\tau\geq 1$ with constant $C_0$ independent of~$\tau$.
\end{definition}

\begin{remark}
For our work below on John domains we could actually assume the weaker condition that $\pp\in \tauLH(\Omega)$ for all $\tau\geq 1$ without the constant being independent of $\tau$.  We make this definition to establish a stronger class that may be needed for future work.
\end{remark}

It is easy to check that $LH_0(\Omega)\subset \partial LH_0(\Omega)$. The following example shows that there are exponents in $\partial LH_0(\Omega)$ that do not belong to $LH_0(\Omega)$.
Let  $\Omega=B(0,1)\subset\R^2$ be the unit disc. Let $A_1$ and $A_2$ be a partition of $\Omega$ into Lebesgue measurable sets (i.e. $A_1\cup A_2=\Omega$ and $A_1\cap A_2=\emptyset$). Let also $1\le p_-<p_+<\infty$. We define $\pp:\Omega\to\R$ as
\[p(x) = \left\{\begin{array}{ll} p_- & \textrm{if }x\in A_1\\
                                                   p_- + (p_ + -p_-)\frac{\log(2)}{-\log(\frac{d(x)}{2})}& \textrm{if } x\in A_2.\end{array}\right.\]
Then, it is easy to check that $\pp\in \partial LH_0(\Omega)$, and it is discontinuous for any $x$ in $\Omega$ that belongs to $\overline{A_1}\cap \overline{A_2}$. Moreover $A_1$ and $A_2$ can be chosen so that $\pp$ is discontinuous at every point in the domain. 

The following lemma gives an equivalent form of property $\partial LH_0^\tau$. It is an analogue to Lemma \cite[Lemma 3.25]{CruzUribeFiorenza} (see also  \cite[Lemma 4.1.6]{DieningBook}).

\begin{lemma}\label{lemma:LHBequiv}
Let $\Omega\subset \R^n$ a domain, $\tau\geq 1$, and $\pp:\Omega\to\R$, $ 1\le p_-\le p_+<\infty$, then the following statements are equivalent: 
\begin{itemize}
\item[(a)] $p(\cdot)\in \tauLH(\Omega)$;
\item[(b)] $|B_{x,\tau}|^{-(p(y)-p_-(B_{x,\tau}))}\le C$ for every $x\in \Omega$ and almost every $y\in B_x$;
\item[(c)] $|B_{x,\tau}|^{-(p_+(B_{x,\tau})-p_-(B_{x,\tau}))}\le C$ for every $x\in \Omega$.
\end{itemize}
\end{lemma}
\begin{proof}

{  We first show that {\rm (a)} implies {\rm (b)}.  Fix $x\in \Omega$ and suppose first that $\tau d(x)>1/2$.  Then 
$|B_{x,\tau}|=c_n (\tau d(x))^n \geq c_n/2^n$.  Therefore,  since for almost every $y \in B_{x,\tau}$, $p(y)\leq p_+(B_{x,\tau})\leq p_+$,
\[ |B_{x,\tau}|^{-(p(y)-p_-(B_{x,\tau}))} \leq (1+2^n/c_n)^{p_+-p_-} = C.
\]
On the other hand, if  $\tau d(x)\leq 1/2$, then 
\begin{multline*}
\log(|B_{x,\tau}|^{p_-(B_{x,\tau})-p(y)}) 
= (p(y)-p_-(B_{x,\tau}))\log(|B_{x,\tau}|^{-1}) \\
\le \frac{C_0}{-\log(\tau d(x))}\log(|B_{x,\tau}|^{-1})
= \frac{C_0}{-\log(\tau d(x))}\log((\tau d(x))^{-n}/c_n)\le C\cdot C_0,
\end{multline*}
where $C$ depends only on the dimension $n$.}

%We first prove that {\rm (a)} implies {\rm (b)}:  for almost every $y\in B_{x,\tau}$, $p(y)\leq p_+(B_{x,\tau})$, and so 

%\begin{align*}
%\log(|B_{x,\tau}|^{p_-(B_{x,\tau})-p(y)}) &= (p(y)-p_-(B_{x,\tau}))\log(|B_{x,\tau}|^{-1}) \\
%&\le \frac{C_0}{-\log(\tau d(x))}\log(|B_{x,\tau}|^{-1})\\ 
%&\le \frac{C_0}{-\log(\tau d(x))}C\log((\tau d(x))^{-n})\le C,
%\end{align*}
%where $C$ depends on $\tau$ and the dimension $n$. 

To prove  that {\rm (b)} implies {\rm (a)}, note that for every $x\in \Omega$ and almost every $y\in B_{x,\tau}$:
\begin{equation*}
\left(\frac{1}{\tau d(x)}\right)^{p(y)-p_-(B_{x,\tau})} 
= (\tau d(x))^{p_-(B_{x,\tau})-p(y)} \\
 \le C |B_{x,\tau}|^{(p(y)-p_-(B_{x,\tau}))/n} \le C, 
\end{equation*}
where $C$ depends on $\tau$, $n$ and $p_+-p_-$. If we take the logarithm, we obtain
\[\big(p(y)-p_-(B_{x,\tau})\big)\log\left(\frac{1}{\tau d(x)}\right) \le \log{C},\]
which is equivalent to  the $\tauLH(\Omega)$ condition.

Since $p(y)\leq p_+(B_{x,\tau})$ for almost every $y \in B_{x,\tau}$, {\rm (c)} implies {\rm (b)}.  To prove the converse, it suffices to note that there exists a sequence $\{y_k\} \subset B_{x,\tau}$ such that {\rm (b)} holds and $p(y_k) \rightarrow p_+(B_{x,\tau})$. If we pass to the limit, we get {\rm (c)}.
\end{proof}

\bigskip

Let us recall that a Whitney decomposition of $\O$ is a collection $\{Q_t\}_{t\in\Gamma}$ of closed dyadic cubes, whose interiors are pairwise disjoint, which satisfies
\begin{enumerate}
\item $\O=\bigcup_{t\in\Gamma}Q_t$,
\item $\diam(Q_t) \leq d(Q_t,\partial\Omega) \leq 4\diam(Q_t)$, \label{Whitney}
\item $\frac{1}{4}\text{diam}(Q_s)\leq \text{diam}(Q_t)\leq 4\text{diam}(Q_s)$, if $Q_s\cap Q_t\neq \emptyset$.
\end{enumerate}

\begin{remark}
Let $Q\subset \Omega$ be a Whitney cube with side length equal to $2^{-k}$. Then, if $\pp \in \tauLH(\Omega)$ for any $\tau \geq 1$, 
\[ p_+(Q) - p_-(Q) \leq \frac{C}{k}. 
\]
%that is, $\ell(Q)=2^{-n}$ and 
%
%\[ \diam(Q)\leq d(Q,\partial \Omega) \leq 4 \diam(Q). \]
%
To see this, let $x_Q$ be the center of $Q$; then 
$d(x_Q) \leq \diam(Q)+\dist(Q,\partial\Omega) \leq 5\diam(Q)=5\sqrt{n}2^{-k}$. 
%$d(x_Q) \geq \dist(Q,\partial\Omega) \geq \diam(Q)=\sqrt{n}2^{-n}$.  
Hence, for any $\tau\geq 1$, $Q\subset B_{x_Q,1} \subset B_{x_Q,\tau}$, and so 
\[ p_+(Q) -p_-(Q) \leq p_+(B_{x_Q,1})-p_-(B_{x_Q,1}) \leq \frac{C_0}{-\log(d(x_Q))} \leq \frac{C_0}{-\log(5\sqrt{n}2^{-k})}\leq \frac{C}{k}. 
\]
\end{remark}

\bigskip

\section{A  Hardy-type operator on variable Lebesgue spaces}
\label{section:hardy}

 Our proof of Theorem~\ref{theorem:sp-general-intro} is based on a local-to-global argument that extends the validity of the inequalities from Whitney cubes to the entire domain $\Omega$ in $\R^n$. We decompose $\Omega$ into a collection of Whitney cubes $\{Q_t\}_{t\in\Gamma}$ and identify two cubes as adjacent if they intersect each other in a $n-1$ dimensional face. It is often helpful to view this discretization of the domain as a graph whose vertices are the Whitney cubes (technically, we consider a small expansion of the cubes) and two cubes are connected by an edge if they are adjacent in the sense given above. We can then define a rooted spanning tree on this graph, where the root is simply a distinguished cube. Usually, we take one of the largest Whitney cubes as the root, but it could be any other cube. This tree structure on the Whitney cubes contains the geometry of the domain, which is fundamental for this analysis. 

To apply this perspective, we recall some definitions and prove some basic results.
A tree is a graph $(V,E)$, where $V$ is the set of vertices and $E$ the set of edges, satisfying that it is connected and has no cycles. A tree is said to be rooted if one vertex is designated as the root. In a rooted tree $(V,E)$, it is possible to define a {\it partial order} ``$\preceq$" in $V$ as follows: $s\preceq t$ if and only if the unique path connecting $t$ to the root $a$ passes through $s$. We write $t \succeq s$ if $s\preceq t$.

{\it The parent} $t_p$ of a vertex $t$ is the vertex connected to $t$ by an edge on the path to the root. It can be seen that each $t\in V$ different from the root has a unique parent, but several elements ({\it children}) in $V$ could have the same parent. Note that two vertices are connected by an edge ({\it adjacent vertices}) if one is the parent of the other. For simplicity, we say that a set of indices $\Gamma$ has a tree structure if $\Gamma$ is the set of vertices of a rooted tree  $(\Gamma,E)$. Also, if the partial order ``$\preceq$" in $\Gamma$ is a total order (i.e. each element in $\Gamma$ has no more than one child), we say that $\Gamma$ is a {\it chain}, or has a chain structure.
For convenience, let us introduce the following notation: \[\Gamma^* = \Gamma\setminus\{a\}.\]

\begin{definition}\label{def:tree-covering}
Let $\Omega\subset\R^n$ be a bounded domain. We say that an open covering $\{U_t\}_{t\in\Gamma}$ is a {\it tree-covering} of $\Omega$ if it also satisfies the properties: 
\begin{enumerate}
\item $\chi_\Omega(x)\leq \sum_{t\in\Gamma}\chi_{U_t}(x)\leq C_1 \chi_\Omega(x)$, for almost every $x\in\Omega$, where $C_1\geq 1$.
\item The set of indices $\Gamma$ has the structure of a rooted tree.
\item There is a collection $\{B_t\}_{t\neq a}$ of pairwise disjoint open sets such that $B_t\subseteq U_t\cap U_{t_p}$, and there is a constant $C_2$ such that: $\frac{|U_t|}{|B_t|}\le C_2$ for every $t\in\Gamma$.
\end{enumerate}
\end{definition}

{
\begin{remark}\label{remark:tree-covering} 
Let $\O$ be a bounded domain and $\{Q_t\}_{t\in\Gamma}$ a Whitney decomposition of it. If we take $U_t$ to be a small expansion of the interior of $Q_t$, for example $U_t=\frac{17}{16}{\mathrm{int}}(Q_t)$, it is clear that we can define a rooted tree structure for $\Gamma$ such that two vertices $s$ and $t$ are adjacent along the tree only if $Q_t\cap Q_s$ share a $n-1$ dimensional face. Hence, every bounded domain admits a tree covering composed of expanded Whitney cubes. 

This tree covering is not unique,  so care should be taken in order to select a tree-covering that contains meaningful information about the geometry of the domain. For example, it is known that the quasi-hyperbolic distance between two cubes in a Whitney decomposition is comparable with the shorter chain of Whitney cubes connecting them (see \cite{H1}). Hence, we can take the open covering $\{U_t\}_{t\in \Gamma}$ and apply an inductive argument to define a tree-covering such that the number of Whitney cubes in each chain to the root is minimal. This tree structure contains some geometric information in terms of the quasi-hyperbolic distance.
\end{remark}
}
For each $t$ we define $W_t$, the \emph{shadow} of $U_t$, to be the set
\[W_t = \bigcup_{s\succeq t} U_s.\]

\bigskip

Let $\Omega\subset\R^n$ be a bounded domain with a tree-covering $\{U_t\}_{t\in\Gamma}$. We define the following Hardy-type operator on the shadows $W_t$:
\begin{equation}\label{def of A}
A_\Gamma f(x) = \sum_{t\in \Gamma^*} \frac{\chi_{B_t}(x)}{|W_t|}\int_{W_t} |f(y)| \dy.
\end{equation}

\begin{theorem}\label{thm: continuity of A}
Let $\Omega\subset\R^n$ be a bounded domain with a tree-covering $\{U_t\}_{t\in\Gamma}$. Suppose $\pp\in \mathcal{P}(\Omega)$ is such that $1<p_-\leq p_+ <\infty$, and  there is a constant $C$ such that 
\begin{equation} \label{eqn:tree-diening}
|W_t|^{-(p(y)-p_-(W_t))}\le C
\end{equation}
for every $t\in \Gamma$ and almost every $y\in W_t$. Then the operator $A_\Gamma$ defined in \eqref{def of A} is bounded from $L^{\pp}(\Omega)$ to itself.
\end{theorem}

%\begin{theorem}\label{thm: continuity of A}
%Let $\Omega\subset\R^n$ be a bounded John domain with a tree-covering $\{U_t\}_{t\in\Gamma}$. Given $\qq\in \mathcal{P}(\Omega)$,  if $\qq \in \partial LH_0^{\tau_K}$ and $1<q_-\leq q_+ <\infty$, then the operator $A_\Gamma$ defined in \eqref{def of A} is bounded from $L^{q(\cdot)}(\Omega)$ to itself.
%\end{theorem}

\begin{proof}

By homogeneity, it suffices to prove that there is a constant $C=C(p(\cdot),\Omega)$ such that

\[\int_\Omega A_\Gamma f(x)^{p(x)}\dx\leq C\]
for any $f\in L^{p(\cdot)}(\Omega)$ with $\|f\|_{L^{p(\cdot)}(\Omega)}=1$. Fix such a function $f$.

Since the sets in the collection $\{B_t\}_{t\neq a}$ are pairwise disjoint we have that
\begin{align*}
\int_\Omega A_\Gamma f(x)^{p(x)}\dx 
&= \int_\Omega \sum_{t\in \Gamma^*}\dfrac{\chi_{B_t}(x)}{|W_t|^{p(x)}}\left(\int_{W_t}|f(y)|\dy\right)^{p(x)}\dx\\ 
&\leq \int_\Omega \sum_{t\in \Gamma^*}\dfrac{\chi_{B_t}(x)}{|W_t|^{p(x)}}\left(\int_{W_t}|f(y)|+1\dy\right)^{p(x)}\dx\\ 
&\leq \int_\Omega \sum_{t\in \Gamma^*}\dfrac{\chi_{B_t}(x)}{|W_t|^{p(x)}}\left(\int_{W_t}\left(|f(y)|+1\right)^{p(y)/p_-(W_t)}\dy\right)^{p(x)}\dx.
\end{align*}

Since $\|f\|_{L^{p(\cdot)}(\Omega)}=1$ and  $p_+<\infty$ we have that
\begin{multline}  \label{estimations}
\int_{W_t}\left(|f(y)|+1\right)^{p(y)/p_-(W_t)}\dy
\leq 
\int_{\Omega}\left(|f(y)|+1\right)^{p(y)/p_-(W_t)}\dy\\ 
\leq 
\int_{\Omega}\left(|f(y)|+1\right)^{p(y)}\dy
\leq 
2^{p_+}\left(\int_{\Omega}|f(y)|^{p(y)}\dy+|\Omega|\right)
\leq 2^{p_+}\left(1+|\Omega|\right). 
\end{multline}

Also, by \eqref{eqn:tree-diening}  there is a constant $C=C(\Omega,p(\cdot))$ such that 
\begin{equation*}
    \dfrac{1}{|W_t|^{p(x)}}\leq \dfrac{C}{|W_t|^{p_-(W_t)}}
\end{equation*}
for all $t\in \Gamma$ and almost every $x\in W_t$. Thus, if we combine the above estimates, we get
\begin{align*}
\int_\Omega A_\Gamma f(x)^{p(x)}\dx 
&\leq C \int_\Omega \sum_{t\in \Gamma^*}\dfrac{\chi_{B_t}(x)}{|W_t|^{p_-(W_t)}}\left(\int_{W_t}\left(|f(y)|+1\right)^{p(y)/p_-(W_t)}\dy\right)^{p_-(W_t)}\dx.
\end{align*}

We want to replace $p_-(W_t)$ by $p_-$ in the previous inequality. If $p_t:=\dfrac{p_-(W_t)}{p_-}=1$, then this is immediate. Otherwise, if $p_t>1$, we apply the classical H\"older inequality in the last integral with exponent $p_t$ to get
\begin{align*}
\int_\Omega A_\Gamma f(x)^{p(x)}\dx 
&\leq C \int_\Omega \sum_{t\in \Gamma*}\dfrac{\chi_{B_t}(x)}{|W_t|^{p_-}}
\left(\int_{W_t}\left(|f(y)|+1\right)^{p(y)/p_-}\dy\right)^{p_-}\dx.
\end{align*}

Finally, again using that the sets  $\{B_t\}_{t\neq a}$ are pairwise disjoint and using the fact that the operator $A_\Gamma$  is bounded on $L^{p_-}(\Omega)$ if $p_->1$, proved by the second author in \cite[Lemma 3.1]{L1}), we get
\begin{align*}
\int_\Omega A_\Gamma f(x)^{p(x)}\dx 
&\leq C \int_\Omega \left(\sum_{t\in \Gamma}\dfrac{\chi_{B_t}(x)}{|W_t|^{p_-}}\int_{W_t}\left(|f(y)|+1\right)^{p(y)/p_-}\dy\right)^{p_-}\dx\\ 
&\leq C \int_\Omega A_\Gamma \left((|f(x)|+1)^{p(y)/p_-}\right)^{p_-}\dx\\ 
&\leq C \int_\Omega (|f(x)|+1)^{p(y)}\dx \\ 
&\leq C(\Omega,p(\cdot)).
\end{align*}
This completes the proof.
\end{proof}

\section{Variable exponent Hardy-type operators and log-H\"older continuity at the boundary}
\label{section:var-Hardy}

In this section we prove several results necessary for our local-to-global argument.  Before doing so, however, we need to explore the relationship between the $\tauLH(\Omega)$ condition and some other averaging-type conditions on the exponent functions.  Our starting point is the observation that many  results which  are obviously true in classical Lebesgue spaces can fail in the variable exponent setting. A very relevant example of this is the so called $K_0$ condition, introduced by Kopaliani~\cite{Kopaliani2007}, which states that 
\[\sup_B |B|^{-1}\|\chi_B\|_\pp\|\chi_B\|_\cpp <\infty,\]
where the supremum is taken over all balls contained in a certain domain. If $\pp$ is constant, the argument of the supremum is exactly $1$ for every ball.  On the contrary, if $\pp$ is not constant, the boundedness of the supremum is not guaranteed: a simple example is given by $p(x)=2+\chi_Q(x)$, where $Q$ is any cube.  (See~\cite[Section~4.4]{CruzUribeFiorenza}.)  Thus, this condition sometimes needs to be imposed on $\pp$ as a hypothesis.

In this section we consider several auxiliary results of this kind, which are crucial in the sequel.  These or similar results are known in the literature given the assumption that the exponent $\pp$ satisfies the global $LH_0(\Omega)$ condition.  However, here we only assume that $\pp \in \partial LH_0^\tau$ for some $\tau\geq 1$; this is sufficient as we will  restrict the analysis to cubes/balls with radius proportional to the distance to the boundary. 

Throughout this section, let $\Omega$ be a bounded domain and $\{U_t\}_{t\in\Gamma}$ a  tree-covering of $\Omega$, like the one given by Remark \ref{remark:tree-covering}. Recall that $U_t$ is an expansion of a Whitney cube $Q_t$, $U_t=\frac{17}{16}{\mathrm{int}}Q_t$. Hence, if we let $x_t$ be the center of $U_t$, then for any $\tau\geq 1$ we have that
\begin{equation}\label{BxtauUt}
U_t\subset B(x_t,\tau d(x_t)) \quad\quad\textrm{and} \quad\quad 
 |U_t|\sim |B(x_t,\tau d(x_t))|,
\end{equation}
where the implicit constants only depend on $n$ and $\tau$. Combined with~Lemma \ref{lemma:LHBequiv}, this implies that if $\pp\in \partial LH^\tau_0$, then
\begin{equation}\label{LHBUt}
    |U_t|^{-(p(y)-p_-(U_t))}\le C    
\end{equation}
for every $t\in \Gamma$ and almost every $y\in U_t$. We will make extensive use of these properties. {  In particular, we will show that properties that hold for balls close to the boundary, as in the $\tauLH(\Omega)$ condition, also hold for the cubes $U_t$ in a tree-covering. Depending on the circumstances, we will emphasize either balls or cubes.}

\bigskip

\begin{lemma}\label{lemma:Tsbounded} Let $\Omega\subset\R^n$ be a bounded domain with a tree-covering $\{U_t\}_{t\in\Gamma}$ as the one given by Remark \ref{remark:tree-covering}. Fix $\pp \in \Pp(\Omega)$, $1\leq p_- \leq p_+<\infty$.  We define the operator
    \[T_{p(\cdot)} f(x) = \sum_{t\in \Gamma} \chi_{U_t}(x)\frac{\|f\chi_{U_t}\|_{p(\cdot)}}{\|\chi_{U_t}\|_{p(\cdot)}} .\]
If $\pp\in \partial LH_0^{\tau}(\Omega)$ for some $\tau\geq 1$, then $T_{p(\cdot)}:L^{p(\cdot)}(\Omega)\to L^{p(\cdot)}(\Omega)$ is bounded. 
\end{lemma}
\begin{proof}
By Fatou's Lemma in the scale of variable Lebesgue spaces (see~\cite[Theorem~2.61]{CruzUribeFiorenza}), we may assume without loss of generality that $f\in L^\infty(\Omega)$, has compact support, and is non-negative. Moreover, a homogeneity argument allows us to assume $\|f\|_{L^p(\cdot)(\Omega)}=1$. Under these assumptions it is enough to find a constant $C$ such that 
\[\int_\Omega T_{p(\cdot)} f(x)^{p(x)} \dx \le C.\]

Since the sets $\{U_t\}_{t\in \Gamma}$ have finite overlap,  we have that
\[\int_\Omega T_{p(\cdot)} f(x)^{p(x)}\dx \le C\sum_{t\in \Gamma} \int_{U_t} \left[\frac{\|\chi_{U_t} f\|_{p(\cdot)}}{\|\chi_{U_t}\|_{p(\cdot)}}\right]^{p(x)} \dx. \]

{  First consider the case
\[\frac{\|\chi_{U_t} f\|_{p(\cdot)}}{\|\chi_{U_t}\|_{p(\cdot)}} \geq 1.\]
This implies $\|\chi_{U_t}\|_{p(\cdot)}\le\|\chi_{U_t} f\|_{p(\cdot)}\le \|f\|_{p(\cdot)} = 1$. Hence, \cite[Corollary 2.23]{CruzUribeFiorenza} yields 
\[ \|\chi_{U_t}\|_{p(\cdot)} \ge |U_t|^{\frac{1}{p_-(U_t)}} 
\quad \text{and} \quad 
\|\chi_{U_t} f\|_{p(\cdot)}^{p_+(U_t)}\le\int_{U_t}f(y)^{p(y)}\dy. \]
If we combine these estimates and \eqref{LHBUt}, we get
\begin{multline*}
\left[\frac{\|\chi_{U_t} f\|_{p(\cdot)}}{\|\chi_{U_t}\|_{p(\cdot)}}\right]^{p(x)} 
\le \frac{\|\chi_{U_t} f\|_{p(\cdot)}^{p_+(U_t)}}{\|\chi_{U_t}\|_{p(\cdot)}^{p_+(U_t)}} 
\le \int_{U_t} f(y)^{p(y)}\dy |U_t|^{-\frac{p_+(U_t)}{p_-(U_t)}}\\
\le C|U_t|^{-1}\int_{U_t} f(y)^{p(y)}\dy
\le C\left(1 + |U_t|^{-1}\int_{U_t}f(y)^{p(y)}\dy\right).
\end{multline*}
On the other hand, if 
\[\frac{\|\chi_{U_t} f\|_{p(\cdot)}}{\|\chi_{U_t}\|_{p(\cdot)}} \leq 1,\]
then it is immediate that the same inequality holds. Therefore, we have that
\begin{multline*}
    \int_\Omega T_{p(\cdot)} f(x)^{p(x)}\dx 
    \le C\sum_{t\in \Gamma} \int_{U_t} \left(1 + |U_t|^{-1}\int_{U_t}f(y)^{p(y)}\dy\right)\dx \\ 
    \le C\sum_{t\in \Gamma}\left(|U_t| + \int_{U_t}f(y)^{p(y)}\dy\right) 
    \le C(|\Omega| + 1).
\end{multline*} }
\end{proof}

{  In Definition~\ref{tauLHB} we generalized  log-H\"older continuity to balls close to the boundary.  Here, we generalize the $K_0$ condition in the same way.}

\begin{definition}
    Given $\pp \in \Pp(\Omega)$, we say that $p(\cdot)$ satisfies the $\partial K_0^\tau(\Omega)$ condition  for some $\tau\geq1$ if 
    \[\sup_{x\in\Omega} |B|^{-1}\|\chi_B\|_{p(\cdot)}\|\chi_B\|_{p'(\cdot)}<\infty,\]
where {  $B=B_{x,\tau}.$}
    Moreover, we say that $\pp$ satisfies the $\partial K_0(\O)$ condition if $\pp\in\partial K_0^\tau(\O)$ for every $\tau\geq 1$ {  with a constant independent of $\tau$.}
\end{definition}

{  Given a tree covering $\{U_t\}_{t\in\Gamma}$, we can define a similar condition,
\begin{equation} \label{eqn:tree-K0}
\sup_{t\in \Gamma} |U_t|^{-1}\|\chi_{U_t}\|_{p(\cdot)}\|\chi_{U_t}\|_{p'(\cdot)}<\infty.
\end{equation}
It follows from~\eqref{BxtauUt} that if $\pp \in \partial K_0^\tau(\Omega)$, then \eqref{eqn:tree-K0} holds.}  The $K_0$ condition is necessary and sufficient to show that averaging operators defined on balls in a domain are bounded; \eqref{eqn:tree-K0} is the same characterization for averaging operators defined on the $U_t$.

\begin{lemma}\label{lemma:ABequivalence}
Given a ball $B$, define the operator $A_{B}$  by
\[A_{B}f(y) = \frac{\chi_{B}(y)}{|B|}\int_{B}f(z)\dz.\]
For any $\pp \in \Pp(\Omega)$ and $\tau \geq 1$,  $p(\cdot)\in\partial K_0^\tau(\Omega)$ if and only if the operators $A_B:L^{p(\cdot)}(\Omega)\to L^{p(\cdot)}(\Omega)$ are uniformly bounded  for every $B=B_{x,\tau}$, {  $x\in \Omega$.}  {  Similarly, given a tree-covering $\{U_t\}_{t\in \Gamma}$, if we define the operators
\[A_{U_t}f(y) = \frac{\chi_{U_t}(y)}{|U_t|}\int_{U_t}f(z)\dz,\]
then the $A_{U_t}$, $t\in \Gamma$, are uniformly bounded if and only \eqref{eqn:tree-K0} holds. }
\end{lemma}
\begin{proof}
The proof is identical to the proof of \cite[Proposition 4.47]{CruzUribeFiorenza}, {  replacing $Q_0$ with $B_{x,\tau}$ or $U_t$.}
\end{proof}

\begin{lemma}\label{lemma:LogHolder-AB}
Given $\pp \in \Pp(\Omega)$ and $\tau \geq 1$, if $p(\cdot)\in\partial LH_0^\tau(\Omega)$, then  $A_B:L^{p(\cdot)}(\Omega)\to L^{p(\cdot)}(\Omega)$ are uniformly bounded for every $B=B_{x,\tau}$, {  $x\in \Omega$.}  {  Similarly, given a tree-covering $\{U_t\}_{t\in \Gamma}$, the operators $A_{U_t}$, $t\in \Gamma$, are uniformly bounded.}
\end{lemma}

\begin{proof}
We prove this for the operators $A_B$; the proof for $A_{U_t}$ is the same, using~\eqref{LHBUt}.
Without loss of generality we can assume that $f$ is non-negative, and by a homogeneity argument we can assume that $\|f\|_{p(\cdot)}=1$.  Then in this case it is enough to show that
\[\int_\Omega |A_B f(y)|^{p(y)}\dy\le C.\]
It follows from the $\partial LH_0^\tau(\Omega)$ condition that
\begin{align*}
    \int_\Omega |A_B f(y)|^{p(y)}\dy &= \int_B \left(\frac{1}{|B|}\int_B f(x) \dz\right)^{p(y)}\dy \\
    &\le \int_B |B|^{-p(y)}\left(\int_B (f(z)+1)\dz\right)^{p(y)}\dy \\
    &\le C\int_B|B|^{-p_-(B)}\left(\int_B (f(z)+1)^{\frac{p(z)}{p_-(B)}}\dz\right)^{p(y)}\dy = I.
\end{align*}
Since $\|f\|_{p(\cdot)}=1$, we have that
\begin{align*}
    \int_B (f(z)+1)^{\frac{p(z)}{p_-(B)}}\dz &\le \int_B (f(z)+1)^{p(z)} \\
    &\le 2^{p_+(\Omega)} \int_B f(z)^{p(z)}\dz + 2^{p_+(\Omega)}|B| \\
    &\le 2^{p_+(\Omega)}(1+|\Omega|); 
\end{align*}
hence, the integral is uniformly bounded above.  Therefore, we can lower the exponent $p(y)$ to $p_-(B)$. Taking this into account and applying the classical H\"older inequality with exponent $p_-(B)$, we obtain
\begin{multline*}
    I 
    \le C \int_B|B|^{-p_-(B)}\left(\int_B (f(z)+1)^{\frac{p(z)}{p_-(B)}}\dz\right)^{p_-(B)}\dy \\
    \le C |B|^{1-p_-(B)}\int_B(f(z)+1)^{p(z)}\dz |B|^{\frac{p_-(B)}{p'_-(B)}} 
    \le C(p(\cdot),\Omega).
\end{multline*}
This concludes the proof.
\end{proof}

\begin{corollary}\label{corollary:LH0impliesK0}
If $p(\cdot)\in \partial LH^\tau_0(\Omega)$, then $p(\cdot)\in\partial K^\tau_0(\Omega)$.
\end{corollary}
\begin{proof} This follows immediately from Lemma \ref{lemma:ABequivalence} and Lemma \ref{lemma:LogHolder-AB}.
\end{proof}

\medskip

{  The following property is again immediate in the constant exponent case.  In variable Lebesgue spaces, or more generally in Banach function spaces, this property, when applied to an arbitrary collection of sets with bounded overlap, is sometimes referred to as Property~G.  See~\cite[Section~7.3]{DieningBook} for details and references.}

 \begin{lemma}\label{lemma:Höldersum}
 Let $\Omega\subset \R^n$ a bounded domain and $\{U_t\}_{t\in\Gamma}$ a tree-covering, like the one given by Remark \ref{remark:tree-covering}. If $\pp \in \Pp(\Omega)$ satisfies $p(\cdot)\in\partial LH_0^{\tau}(\Omega)$ for some $\tau\ge 1$, then  for every $f\in L^{\pp}(\O)$ and  $g\in L^{\cpp}(\O)$ we have that
\begin{equation}\label{holdersum}
\sum_{t \in \Gamma} \|\chi_{U_t} f\|_{p(\cdot)} \|\chi_{U_t} g\|_{p'(\cdot)} \le C \|f\|_{p(\cdot)}\|g\|_{p'(\cdot)}.
\end{equation}
 \end{lemma}
\begin{proof}
Using \eqref{eqn:tree-K0} we have that
   \begin{equation*}
   \sum_{t \in \Gamma} \|\chi_{U_t} f\|_{p(\cdot)} \|\chi_{U_t} g\|_{p'(\cdot)} 
   \le C\sum_{t \in \Gamma} |U_t| \frac{\|\chi_{U_t} f\|_{p(\cdot)}}{\|\chi_{U_t}\|_\pp} 
   \frac{\|\chi_{U_t} g\|_{p'(\cdot)}}{\|\chi_{U_t}\|_\cpp } =: I.
   \end{equation*}
   Since the sets $\{U_t\}_{t\in \Gamma}$ have finite overlap, we have that 
  {  \[  \sum_{t \in \Gamma} \chi_{U_t}(x) \frac{\|\chi_{U_t} f\|_{p(\cdot)}}{\|\chi_{U_t}\|_\pp} \frac{\|\chi_{U_t} g\|_{p'(\cdot)}}{\|\chi_{U_t}\|_\cpp } \leq CT_\pp f(x) T_\cpp g(x).  \] }
  If we integrate this estimate over $\Omega$, apply   H\"older's inequality \eqref{eqn:holder}, and then apply Lemma~\ref{lemma:Tsbounded}, we get
    \begin{multline*}
        I 
        \le C \int_{\Omega} T_\pp f(x) T_\cpp g(x) \dx \\
          \le C \|\chi_{\Omega} T_\pp f\|_\pp\|\chi_{\Omega}T_\cpp g\|_\cpp 
          \le C \|\chi_{\Omega}f\|_\pp\|\chi_{\Omega}g\|_\cpp.
    \end{multline*}
    This completes the proof.
\end{proof}

\medskip 

\begin{remark} In Lemma~\ref{lemma:Höldersum} we proved \eqref{holdersum} using the continuity of the operator $T_\pp$, which holds thanks to the hypothesis $\pp\in \partial LH_0^{\tau}(\Omega)$ for some $\tau\geq 1$. It is interesting to notice that the reverse implication is also true. For if~\eqref{holdersum} holds, then by \eqref{eqn:assoc-norm},  H\"older's inequality \eqref{eqn:holder}, and \eqref{holdersum}, we get 
\begin{align*}
   \|T_\pp f \|_{p(\cdot)} &\le C\sup_{g:\|g\|_{p'(\cdot)}\le 1} \int_\Omega \sum_{t \in \Gamma} \chi_{U_t}(x) \frac{\|\chi_{U_t} f\|_{p(\cdot)}} {\|\chi_{U_t}\|_{p(\cdot)}} g(x)\,\dx \\
    &\le C \sup_{g:\|g\|_{p'(\cdot)}\le 1}\sum_{t \in \Gamma} \int_{U_t} \frac{\|\chi_{U_t} f\|_{p(\cdot)}} {\|\chi_{U_t}\|_{p(\cdot)}} |g(x)|\,\dx \\
    &\le C \sup_{g:\|g\|_{p'(\cdot)}\le 1}\sum_{t \in \Gamma} \|\chi_{U_t} f\|_{p(\cdot)} \|\chi_{U_t} g\|_{p'(\cdot)} \\
    &\le C \sup_{g:\|g\|_{p'(\cdot)}\le 1} \|f\|_{p(\cdot)} \|g\|_{p'(\cdot)} \\
    &\le C \|f\|_{p(\cdot)}.
\end{align*}
\end{remark}

\medskip

Finally we prove the following lemma, which resembles~\cite[Theorem~4.5.7]{DieningBook}.

\begin{lemma}\label{corollary:measUt-meanp}
   Given $\pp \in \Pp(\Omega)$, suppose $1<p_-\leq p_+<\infty$ and $\pp\in \partial LH_0^\tau(\O)$.  If $\{U_t\}_{t\in\Gamma}$ is a tree-covering of $\Omega$, like the one given by Remark \ref{remark:tree-covering}, and \[\frac{1}{p_{U_t}}=\frac{1}{|U_t|}\int_{U_t}\frac{1}{p(x)}\dx,\] then for every $t\in \Gamma$,
    \begin{equation} \label{eqn:mm1}
\|\chi_{U_t}\|_\pp \sim |U_t|^\frac{1}{p_{U_t}}.
\end{equation}
    where the implicit constants are independent of $t\in \Gamma$.
\end{lemma}

\begin{proof}
Recall that as in Remark~\ref{remark:tree-covering} each set $U_t$ is cube.  
The inequality
\begin{equation} \label{eqn:measure-norm1}
|U_t|^\frac{1}{p_{U_t}} \le 2 \|\chi_{U_t}\|_\pp
\end{equation}
was proved in \cite[Lemma 3.1]{Cruz_Uribe_2016} for every $\pp\in\Pp(\Omega)$ with $1<p_-\le p_+<\infty$ and for every cube $Q\subset \Omega$. 

To prove the reverse inequality, we apply a duality argument. By \eqref{eqn:assoc-norm}, there exists  $g\in L^\cpp(\O)$, $\|g\|_\cpp=1$, such that
\begin{align*}
\frac{\|\chi_{U_t}\|_\pp}{|{U_t}|^\frac{1}{p_{U_t}}} 
&\le C \frac{1}{|{U_t}|^\frac{1}{p_{U_t}}}\int_\O \chi_{U_t}(x) g(x)\dx \\
&= C|{U_t}|^\frac{1}{p'_{U_t}}\frac{1}{|{U_t}|}\int_{U_t} g(x)\dx. \\
\intertext{But by inequality~\eqref{eqn:measure-norm1} we have that}
&\le C\|\chi_{U_t}\|_\cpp \frac{1}{|{U_t}|}\int_{U_t} g(x)\dx \\
&= C\|A_{U_t} g\|_\cpp \\
& \leq C \|g\|_\cpp \\
& = C.
\end{align*}
Here $A_{U_t}$ is the operator defined in Lemma \ref{lemma:ABequivalence}.  The last inequality follows, since by Lemma~\ref{lemma:dual-LHB}, $\cpp\in \partial LH_0^\tau(\O)$, and so the $A_{U_t}$ are uniformly bounded on $L^\cpp(\Omega)$.
\end{proof}

\medskip

{  
In order to prove Sobolev-Poincaré inequalities, we will need off-diagonal versions of some of the previous results. We will state these results in more generality than we need to prove the results in Section~\ref{section:improved-SPi}.  We do so partly because of the intrinsic interest of these results, but also to highlight where the  more restrictive hypotheses are needed.  

One natural set of assumptions would be to fix a value $\alpha$, $0<\alpha<n$, let $\pp \in \Pp(\Omega)$ satisfy $1\leq p_-\leq p_+<n/\alpha$, and define $\qq \in \Pp(\Omega)$ pointwise by the identity
\begin{equation} \label{eqn:qq-defn}
 \frac{1}{p(x)} - \frac{1}{q(x)} = \frac{\alpha}{n}.  
 \end{equation}
%
%Note that with this assumption, if we argue as we did in the proof of Lemma~\ref{lemma:dual-LHB}, we immediately have that if $\pp \in \tauLH(\Omega)$ for some $\tau\geq 1$, then $\qq\in \tauLH(\Omega)$ as well.

However, we would like to avoid the restriction that $p_+<n/\alpha$.  To do so, we will assume that $\pp, \qq \in \P(\Omega)$ satisfy the following:
\begin{equation} \label{eqn:pq-assump1}
1\leq p_-\leq p_+<\infty, \qquad 1 \leq q_- \leq q_+ < \infty, 
\end{equation}
and that there exists $0<\beta<\alpha$ such that
\begin{equation} \label{eqn:pq-assump2}
\frac{\beta}{n} = \frac{1}{p(x)} - \frac{1}{q(x)} \leq \frac{\alpha}{n}. 
 \end{equation}
\begin{remark} \label{remark:pp-qq-LH}
    If we argue as we did in the proof of Lemma~\ref{lemma:dual-LHB}, if $\pp \in \tauLH(\Omega)$ for some $\tau\geq 1$, then \eqref{eqn:pq-assump1} and the first inequality in \eqref{eqn:pq-assump2} imply that $\qq \in \tauLH(\Omega)$.
\end{remark}

\begin{remark}
    Given any $\pp \in \Pp(\Omega)$ such that $p_+<\infty$, and given any $0<\alpha<n$, there exists $\qq\in \Pp(\Omega)$ and $0<\beta<\alpha$ such that~\eqref{eqn:pq-assump2} holds.   If $p_+<n/\alpha$, let $\beta=\alpha$ and use the first inequality in~\eqref{eqn:pq-assump2} to define $\qq$.  If $n/\alpha \leq p_+<\infty$, fix $0<\beta<\alpha $ such that $p_+<n/\beta$ and again define $\qq$ using~\eqref{eqn:pq-assump2}.
\end{remark}
}

{  
We first consider an off-diagonal operator $T_\pp^\alpha$, similar to the one defined in Lemma~\ref{lemma:Tsbounded}.
}

{  
\begin{lemma}\label{lemma:Talphabounded}
Let $\O\subset \R^n$ a bounded domain with a tree-covering $\{U_t\}_{t\in\Gamma}$, like the one given by Remark~\ref{remark:tree-covering}.  Fix $\alpha$, $0<\alpha<n$, and $\pp,\, \qq  \in \Pp(\Omega)$ that satisfy \eqref{eqn:pq-assump1} and~\eqref{eqn:pq-assump2} for some $0<\beta<\alpha$.  Define the operator 
\[T^\alpha_\pp f(x) = \sum_{t\in \Gamma} |U_t|^\frac{\alpha}{n} \frac{\|\chi_{U_t} f\|_\pp}{\|\chi_{U_t}\|_\pp}\chi_{U_t}(x).\]

If $\pp \in \tauLH(\Omega)$ for some $\tau\geq 1$, then $T^\alpha_\pp:L^\pp(\O)\to L^\qq(\O)$ is bounded.
\end{lemma}

\begin{proof}
    As in the proof of Lemma \ref{lemma:Tsbounded}, it will suffice  to prove that for $f\in L^\pp(\Omega)$ such that $f\ge 0$ and $\|f\|_\pp=1$, there is a constant $C$ such that
    \[\int_\Omega T^\alpha_\pp f(x) ^{q(x)} \dx \le C.\]
    
    By  Corollary~\ref{corollary:measUt-meanp}, since $\pp \in \tauLH(\Omega)$ we have that 
\[  \frac{|U_t|^{\frac{\alpha}{n}}}{\|\chi_{U_t}\|_\pp}
\approx \frac{|U_t|^{\frac{\alpha}{n}}}{|U_t|^{\frac{1}{p_{U_t}}}} 
= |U_t|^{-\frac{1}{q_{U_t}}} |U_t|^{\frac{\alpha}{n} - \frac{1}{p_{U_t}} + \frac{1}{q_{U_t}}}.  
\]
If we integrate~\eqref{eqn:pq-assump2} over $U_t$, we see that
\begin{equation} \label{eqn:pq-assump3}
    \frac{\alpha}{n} - \frac{1}{p_{U_t}} + \frac{1}{q_{U_t}} \geq 0.  
\end{equation} 
Hence, again by  Corollary~\ref{corollary:measUt-meanp}, since $\qq \in \tauLH(\Omega)$ by Remark~\ref{remark:pp-qq-LH}, 
\[ |U_t|^{-\frac{1}{q_{U_t}}} |U_t|^{\frac{\alpha}{n} - \frac{1}{p_{U_t}} + \frac{1}{q_{U_t}}}
\lesssim  \frac{1}{\|\chi_{U_t}\|_\qq} |\Omega |^{\frac{\alpha}{n} - \frac{1}{p_-} + \frac{1}{q_+}}. 
\]

If we combine these two inequalities, using $q_+<\infty$, we have that
    \begin{multline*}
    \int_\Omega T^\alpha_\pp f(x) ^{q(x)} \dx \\
    = \sum_{t\in\Gamma} \int_{U_t} \left[\frac{\|\chi_{U_t} f\|_\pp}{\|\chi_{U_t}\|_\pp}|U_t|^\frac{\alpha}{n}\right]^{q(x)}\dx 
    \le C^{q_+}\sum_{t\in\Gamma}\int_{U_t}\left[\frac{\|\chi_{U_t} f\|_\pp}{\|\chi_{U_t}\|_\qq}\right]^{q(x)}\dx;
    \end{multline*}
    The constant $C$ depends on $\pp$, $\qq$, $\alpha$, $n$, and $|\Omega|$.  
    The argument continues as in the proof of Lemma~\ref{lemma:Tsbounded}. Since $\|f\|_\pp=1$, if the expression in square brackets is bigger than $1$,  then $\|\chi_{U_t}\|_\qq\le 1$, so by~\cite[Corollary 2.23]{CruzUribeFiorenza} and the fact that $p_+(U_t)\le q_+(U_t)$ (which follows from \eqref{eqn:pq-assump3})  we have that
    \begin{multline*}
    \left[\frac{\|\chi_{U_t} f\|_\pp}{\|\chi_{U_t}\|_\qq}\right]^{q(x)} 
    \le \left[\frac{\|\chi_{U_t} f\|_\pp}{\|\chi_{U_t}\|_\qq}\right]^{q_+(U_t)} \\
    \le \frac{\|\chi_{U_t} f\|_\pp^{p_+(U_t)}}{\|\chi_{U_t}\|_\qq^{q_+(U_t)}}
    \le C\frac{\int_{U_t}f(y)^{p(y)}\dy}{|U_t|^{\frac{q_+(U_t)}{q_-(U_t)}}} 
    \le C \int_{U_t}f(y)^{p(y)}\dy |U_t|^{-1},
    \end{multline*}
    where in the last step we used Lemma \ref{lemma:LHBequiv} and the fact that $\qq \in \tauLH(\Omega)$. Hence, in any case we have that
       \[\left[\frac{\|\chi_{U_t} f\|_\pp}{\|\chi_{U_t}\|_\qq}\right]^{q(x)} \le 1 +   C \int_{U_t}f(y)^{p(y)}\dy |U_t|^{-1},\]
    and consequently,
    \begin{multline*}
     \int_\Omega T^\alpha_\pp f(x) ^{q(x)} \dx 
       \le C\sum_{t\in\Gamma} \int_{U_t}\left[1+C\int_{U_t}f(y)^{p(y)}\dy|U_t|^{-1}\right]\dx\\
        \le C\sum_{t\in\Gamma} |U_t|+C\int_{U_t}f(y)^{p(y)}\dy
        \le C + C|\Omega|,
    \end{multline*}
    which completes the proof.
\end{proof}

\bigskip

}

{  
The following is an off-diagonal version of the $\partial K_0^\tau$ condition.  For all cubes or balls in a given domain it was introduced by the first author and Roberts~\cite{dcu-roberts}.

\begin{lemma}\label{lemma:offdiag-partialK0}
Let $\O\subset\R^n$ be a bounded domain with a tree-covering $\{U_t\}_{t\in\Gamma}$, like the one given by Remark \ref{remark:tree-covering}. Fix $\alpha$, $0<\alpha<n$, and $\pp,\, \qq  \in \Pp(\Omega)$ that satisfy \eqref{eqn:pq-assump1} and~\eqref{eqn:pq-assump2} for some $0<\beta\le\alpha$. Suppose $\pp \in \partial LH_0^\tau(\O)$ for some $\tau\ge 1$.  Then the following estimates hold: 
{ 
\begin{gather}
    \sup_{t\in \Gamma} |U_t|^{-1+\frac{\alpha}{n}}\|\chi_{U_t}\|_\qq\|\chi_{U_t}\|_\cpp<\infty, \label{eqn:offdiag-K0-1}\\
    \sup_{t\in \Gamma} |U_t|^{-1-\frac{\beta}{n}}\|\chi_{U_t}\|_\cqq\|\chi_{U_t}\|_\pp<\infty.\label{eqn:offdiag-K0-2}
\end{gather}
}
\end{lemma} 
}

\begin{remark}
    As will be clear from the way the proof is written, for \eqref{eqn:offdiag-K0-2} to be true,  instead of the first equality in~\eqref{eqn:pq-assump2}, it suffices to assume that $\frac{\beta}{n} \leq \frac{1}{p(x)}-\frac{1}{q(x)}$.
\end{remark}
{ 
\begin{proof}
We first  prove~\eqref{eqn:offdiag-K0-1}.  Fix $t\in \Gamma$; then by Lemma~\ref{lemma:dual-LHB} and Remark~\ref{remark:pp-qq-LH}, $\cpp,\, \qq \in \tauLH(\Omega)$. Hence, by Lemma~\ref{corollary:measUt-meanp}, we get
\begin{multline*}
    |U_t|^{-1+\frac{\alpha}{n}}\|\chi_{U_t}\|_\qq\|\chi_{U_t}\|_\cpp
     \approx  |U_t|^{-1+\frac{\alpha}{n}}|U_t|^{\frac{1}{q_{U_t}}}|U_t|^{\frac{1}{p_{U_t}'}} \\
     = |U_t|^{-1+\frac{\alpha}{n}}|U_t|^{\frac{1}{q_{U_t}}}|U_t|^{1-\frac{1}{p_{U_t}}} 
     = |U_t|^{\frac{\alpha}{n}+\frac{1}{q_{U_t}}-\frac{1}{p_{U_t}}}
     \leq |\Omega|^{\frac{\alpha}{n}+\frac{1}{q_-}-\frac{1}{p_+}}.
\end{multline*}
The last inequality follows from~\eqref{eqn:pq-assump3}.
The implicit constants are independent of $t$, so if we take the supremum, we get the desired estimate.  

The proof of~\eqref{eqn:offdiag-K0-2} is nearly identical.  Since $\pp,\,\cqq\in \tauLH(\Omega)$, for any $t\in\Gamma$, by 
\begin{multline*}
    |U_t|^{-1-\frac{\beta}{n}}\|\chi_{U_t}\|_\cqq\|\chi_{U_t}\|_\pp
     \approx  |U_t|^{-1-\frac{\beta}{n}}|U_t|^{\frac{1}{q_{U_t}'}}|U_t|^{\frac{1}{p_{U_t}}} \\
         = |U_t|^{-\frac{\beta}{n}-\frac{1}{q_{U_t}}+\frac{1}{p_{U_t}} }
         \leq |\Omega|^{-\frac{\beta}{n}-\frac{1}{q_+}+\frac{1}{p_-}};
     \end{multline*}
     the last inequality follows from the inequality we get if we integrate  the lower estimate in~\eqref{eqn:pq-assump2} as we did to find~\eqref{eqn:pq-assump3}.
\end{proof}
}
% \begin{align*}
% |U_t|^{-1+\frac{\alpha}{n}} \|\chi_{U_t}\|_\qq\|\chi_{U_t}\|_\cpp &\approx |U_t|^{-1+\frac{\alpha}{n}}|U_t|^{\frac{1}{q_{U_t}}}|U_t|^{\frac{1}{p'_{U_t}}}\\
% & = |U_t|^{-1+\frac{\alpha}{n}+\frac{1}{q_-(U_t)}-\frac{1}{p_+(U_t)}}|U_t|^{\frac{1}{q_{U_t}}-\frac{1}{q_-(U_t)}}|U_t|^{1-\frac{1}{p_{U_t}}+\frac{1}{p_+(U_t)}} \\
% \intertext{We continue by applying \ref{lemma:LHBequiv} and the definition of $\qq$:}
% &\le C |U_t|^{\frac{\alpha}{n}+\frac{1}{q_-(U_t)}-\frac{1}{p_+(U_t)}} \\
% &\le C |U_t|^{\frac{\alpha}{n}+\frac{1}{q_-(U_t)}-\frac{1}{p_-(U_t)}} |U_t|^{\frac{1}{p_-(U_t)}-\frac{1}{p_+(U_t)}}\\
% &\le C|\O|^{\frac{1}{p_-(\O)}-\frac{1}{p_+(\O)}}.
% \end{align*}
% The second estimate can be proven by the same means, but it is enough to observe that the relation $\frac{1}{\pp}-\frac{1}{\qq}=\frac{\alpha}{n}$ is reversed for the conjugate exponents: $\frac{1}{\cqq}-\frac{1}{\cpp} = \frac{\alpha}{n}$, which leads to the change of sign in the exponent $\frac{\alpha}{n}$ on the left hand side.  

Finally, we prove an off-diagonal analog of Lemma \ref{lemma:Höldersum}.  It is only in this result, which requires~\eqref{eqn:offdiag-K0-2}, that we were required to introduce the parameter $\beta$.  In Section~\ref{section:improved-SPi} below, we will choose the parameter $\alpha$ so that we have $\beta=\alpha$ when applying this result. 

{ 
\begin{lemma}\label{lemma:offdiag-Höldersum}
 Let $\Omega\subset \R^n$ a bounded domain and $\{U_t\}_{t\in\Gamma}$ a tree-covering, like the one given by Remark \ref{remark:tree-covering}. Fix $\alpha$, $0<\alpha<n$, and $\pp,\, \qq  \in \Pp(\Omega)$ that satisfy \eqref{eqn:pq-assump1} and~\eqref{eqn:pq-assump2} for some $0<\beta\le\alpha$.  Suppose  that $p(\cdot) \in\partial LH_0^{\tau}(\Omega)$ for some $\tau\ge 1$.  Then  for every $f\in L^{\pp}(\O)$ and  $g\in L^{\cqq}(\O)$, the following inequality holds:
\begin{equation}\label{offdiag-holdersum}
\sum_{t \in \Gamma} \|\chi_{U_t} f\|_{p(\cdot)} \|\chi_{U_t} g\|_{q'(\cdot)} \le C \|f\|_{p(\cdot)}\|g\|_{q'(\cdot)}.
\end{equation}
\end{lemma}

\begin{proof}
% Using Lemma \ref{lemma:offdiag-partialK0}, \eqref{eqn:qq-defn} and the fact that $|U_t|$ is bounded we have that:
% \begin{align*}
% |U_t|^{\frac{1-\alpha}{n}} &\approx |U_t|^\frac{1-\alpha}{n}|U_t|^{\frac{1}{p_{U_t}}+1-\frac{1}{q_{U_t}}} \|\chi_{U_t}\|^{-1}_\pp \|\chi_{U_t}\|^{-1}_\cqq \\
% &=|U_t|^\frac{1-\alpha}{n}|U_t|^{\frac{1}{p_{U_t}}+1-\frac{1}{p_{U_t}}+\frac{\alpha}{n}} \|\chi_{U_t}\|^{-1}_\pp \|\chi_{U_t}\|^{-1}_\cqq \\
% &= |U_t|^{1+\frac{1}{n}}\|\chi_{U_t}\|^{-1}_\pp \|\chi_{U_t}\|^{-1}_\cqq \\
% &\le |U_t|^{1+\frac{\alpha}{n}}\|\chi_{U_t}\|^{-1}_\pp \|\chi_{U_t}\|^{-1}_\cqq.
% \end{align*}
The proof is very similar to the proof of Lemma~\ref{lemma:Höldersum}, so we omit some details.  By inequality~\ref{eqn:offdiag-K0-2} and since the cubes $\{U_t\}_{t\in \Gamma}$ have finite overlap, we have that
\begin{align*}
    \sum_{t \in \Gamma} \|\chi_{U_t} f\|_{p(\cdot)} \|\chi_{U_t} g\|_{q'(\cdot)} 
    &\le C\sum_{t \in \Gamma} |U_t|^{1+\frac{\beta}{n}}\frac{\|\chi_{U_t} f\|_{p(\cdot)}}{\|\chi_{U_t}\|_\pp} \frac{\|\chi_{U_t} g\|_{q'(\cdot)}}{\|\chi_{U_t}\|_\cqq} \\
    &\le C \int_\Omega T_{\pp}^\beta f(x) T_\cqq g(x) \dx\\
    &\le C\|T_\pp^\beta f\|_\qq \|T_\cqq g\|_\cqq\\
    &\le C\|f\|_\pp \|g\|_\cqq ;
\end{align*}
the last two inequalities follow from H\"older's inequality \ref{eqn:holder}, and from Lemmas~\ref{lemma:Tsbounded} and~\ref{lemma:Talphabounded}.
\end{proof}
}

\section{A decomposition of functions for John domains}
\label{section:john}

In this section we prove a decomposition theorem  which in our local-to-global argument will let us extend our results from cubes to John domains.  We begin by recalling the definition of John domains.  They  were introduced by Fritz John in \cite{J}. They include domains with a fractal boundary, such as the interior of the Koch snowflake, domains with inner cusps and even with cuts, etc. However, they have properties similar to more \emph{regular} domains with respect to  Sobolev-Poincaré type inequalities.  It was shown by Bojarski~\cite{MR0982072} that Sobolev-Poincar\'e inequalities hold on John domains without weights. Later, Chua~\cite{MR1140667} showed that a weighted Sobolev-Poincar\'e inequality holds, where the same weight appears on the left and right hand sides. By contrast, the classical Sobolev-Poincaré inequality does not hold on domains with external cusps or, more generally, with boundary of type H\"older-$\alpha$. For these domains, \emph{weighted} versions of the inequalities can be derived, with a weight that compensates the singularities of the boundary.

\begin{definition}\label{definition:john}
A bounded domain $\Omega$ in $\R^n$ is a John domain with parameter $\lambda>1$ if there exists a point $x_0\in\Omega$ such that, given any $y\in\Omega$, there exists a rectifiable curve parameterized by arc length $\gamma : [0,\ell] \rightarrow \Omega$, with $\gamma(0)=y$, $\gamma(\ell)=x_0$, and $\lambda\, \dist(\gamma(t),\partial \Omega) \geq t$.
\end{definition}

The following result was proved by the second author in \cite{L2} and gives a characterization of John domains in terms of tree-coverings.

\begin{proposition}\label{prop:tcov-John} A bounded domain $\O \subset \R^n$ is a John domain if and only if given a Whitney decomposition $\{Q_t\}_{t\in\Gamma}$ of $\O$, there exists a tree structure for the set of indices $\Gamma$ satisfying the conditions in Remark \ref{remark:tree-covering} and a constant $K>1$ such that
\begin{align}\label{Boman tree}
Q_s\subseteq KQ_t,
\end{align}
for any $s,t\in\Gamma$ with $s\succeq t$. In other words, the shadow $W_t$ of $Q_t$ is contained in $KQ_t$. 
%the intersection of the cubes associated to adjacent indices, $Q_t$ and $Q_{t_p}$, is an $n-1$ dimensional face of one of these cubes.
\end{proposition}

\begin{remark} \label{remark:tree-covering-john}
Let $\O\subset \R^n$ be a bounded John domain and $\{Q_t\}_{t\in\Gamma}$ a Whitney decomposition of $\O$. Let $x_t$ denote the center of $Q_t$. Then by Proposition \ref{prop:tcov-John}, there is a constant $\tau_K$ depending on $K$ such that $W_t\subset B_{x_t,\tau_K}$. Moreover, $|B_{x_t,\tau_K}|\sim |W_t|$ with constants depending on $n$ and $K$. {  Hence, for John domains, the characterization of log-H\"older continuity at the boundary in Lemma \ref{lemma:LHBequiv} implies that $\pp \in \partial LH_0^{\tau_K}(\Omega)$ if and only if for every $t\in\Gamma$,
\[|W_t|^{-(p(y)-p_-(W_t))} \le C. \]
In fact, this estimate is the main assumption on $\pp$ that we actually need in the sequel. }
\end{remark}

Hereafter, when we consider a tree-covering of a John domain, we assume that the tree-covering is taken as in Proposition \ref{prop:tcov-John} and that $\tau_K$ is the constant from Remark~\ref{remark:tree-covering-john}.   As an immediate consequence, we get the following corollary to Theorem~\ref{thm: continuity of A}.

\begin{corollary}\label{continuity of A on John}
Let $\Omega\subset\R^n$ be a bounded John domain with a tree-covering $\{U_t\}_{t\in\Gamma}$. If $\qq\in \mathcal{P}(\Omega)$ is such that $\qq \in \partial LH_0^{\tau_K}(\Omega)$ condition and $1<q_-\leq q_+ <\infty$, then the operator $A_\Gamma$ defined in \eqref{def of A} is bounded from $L^{q(\cdot)}(\Omega)$ to itself.
\end{corollary}

{  Our main result in this section is the following decomposition theorem.  We begin with a definition.

\begin{definition}\label{Definition decomposition} 
Given a bounded domain $\Omega\subset\R^n$, let  $\{U_t\}_{t\in\Gamma}$  be a tree-covering of $\Omega$ as in Remark~\ref{remark:tree-covering}.    Given $g\in L^1(\Omega)$ with $\int g \dx =0$, we say that a collection of functions $\{g_t\}_{t\in\Gamma}$
  in $L^1(\Omega)$ is a \emph{decomposition of $g$} subordinate to $\{U_t\}_{t\in\Gamma}$ if the following  properties are satisfied:
\begin{enumerate}
\item $\displaystyle g=\sum_{t\in \Gamma} g_t;$
\item $\supp (g_t)\subset U_t;$
\item $\displaystyle \int_{U_t} g_t \dx =0$ for all $t\in\Gamma$.
\end{enumerate}
\end{definition}  }

\medskip

\begin{theorem}\label{Decomp Thm} 
Given a bounded John domain $\Omega\subset\R^n$, let $\{U_t\}_{t\in \Gamma}$ be a tree-covering of $\O$ as given by Proposition~\ref{prop:tcov-John}. Fix $\qq\in \partial LH_0^{\tau_K}(\O)$, with $1<q^-\le q^+<\infty$. Then for every $g\in L^{q(\cdot)}(\Omega)$ such that $\int_\O g\dx =0$, there exists a decomposition $\{g_t\}_{t\in\Gamma}$ of $g$ subordinate to $\{U_t\}_{t\in\Gamma}$ with the additional property that
\begin{equation}\label{Decomp estim}
\bigg\|\sum_{t\in\Gamma} \chi_{U_t} \frac{\|\chi_{U_t} g_t\|_{q(\cdot)}}{\|\chi_{U_t}\|_{q(\cdot)}}\bigg\|_{q(\cdot)} \le C \|g\|_{q(\cdot)}.
\end{equation}
\end{theorem}
\begin{proof}
Fix $g\in L^{q(\cdot)}(\Omega)$.  Let $\{\phi_t\}_{t\in\Gamma}$ be a partition of the unity subordinate to $\{U_t\}_{t\in\Gamma}$: i.e.,  ${\rm{supp}}(\phi_t)\subset U_t$, $0\le\phi_t(x)\le 1$, and $\sum_t \phi_t(x) = 1$ for all $x\in\Omega$. We first define an initial decomposition of $g$  by $f_t = g\phi_t$. The collection $\{f_t\}_{t\in\Gamma}$ satisfies properties $(1)$ and $(2)$ in Definition~\ref{Definition decomposition}, but not necessarily $(3)$. Hence, we modify these functions as follows.
For each $s\in \Gamma$, $s\neq a$, define
\[h_s(x) := \frac{\chi_s(x)}{|B_s|}\int_{W_s} \sum_{t\succeq s} f_t(y)\dy,\]
Note that ${\rm{supp}}(h_s)\subset B_s$ and 
$$\int h_s(x)\dx = \int_{W_s}\sum_{t\succeq s} f_t(y)\dy.$$ 
Now define, for $t\neq a$, 
\begin{equation*}
g_t(x) := f_t(x) + \Big(\sum_{s:s_p = t}h_s(x)\Big) - h_t(x),
\end{equation*}
and 
\[ g_a(x) := f_a(x) + \sum_{s:s_p = a}h_s(x). \]
Recall that $s_p$ denotes the parent of $s$ in the tree $\Gamma$.
Note that the summations in these definitions are finite since they are indexed over the children of $t$ (or $a$). It is easy to check that $\{g_t\}_{t\in\Gamma}$ satisfies all the properties of Definition \ref{Definition decomposition}. Therefore, we only need to  prove~\eqref{Decomp estim}. 

Observe that for any $s,t\in \Gamma$, $s\neq a$, 
\begin{align*}
|h_s(x)| &\le \frac{\chi_{s}(x)}{|B_s|} \int_{W_s} |g(x)|\,\dx \le \frac{|W_s|}{|B_s|}\frac{\chi_{s}(x)}{|W_s|} \int_{W_s} |g(x)|\,\dx,
\end{align*}
and
\[|f_t(x)| \le |g(x)|\chi_{U_t}(x).\]
Moreover, by Proposition \ref{prop:tcov-John}, we have that $|W_s|\le C|B_s|$,  where the constant $C$ depends only on  $n$ and the constant in \eqref{Boman tree}.  Therefore, 
\[|g_t(x)| \le C\Big(\chi_{U_t} |g(x)| + \sum_{s:s_p=t} \frac{\chi_s(x)}{|W_s|}\int_{W_s} |g|\dy + \frac{\chi_t(x)}{|W_t|}\int_{W_t} |g|\,dy \Big),\]
and so we can estimate as follows:
\begin{align*}
\bigg\|\sum_{t\in\Gamma} \chi_{U_t} \frac{\|g_t\|_\qq} {\|\chi_{U_t}\|_\qq} \bigg\|_\qq 
&\le 
C\bigg\{\bigg\|\sum_{t\in\Gamma} \chi_{U_t} \frac{\|g\chi_{U_t}\|_\qq} {\|\chi_{U_t}\|_\qq} \bigg\|_\qq  \\
&\qquad + \bigg\|\sum_{t\in\Gamma} \chi_{U_t} \frac{\|\sum_{s:s_p=t} \frac{\chi_s(x)}{|W_s|}\int_{W_s} |g|\|_{\qq}} {\|\chi_{U_t}\|_\qq} \bigg\|_{\qq} \\
&\qquad +  \bigg\|\sum_{t\in\Gamma} \chi_{U_t} \frac{\| \frac{\chi_t(x)}{|W_t|}\int_{W_t} |g|\|_\qq} {\|\chi_{U_t}\|_\qq} \bigg\|_\qq \bigg\} \\
&= C\{I_1 + I_2 + I_3\}.
\end{align*}
By Lemma \ref{lemma:Tsbounded},
\[I_1 = \|T_\qq g\|_\qq \le C\|g\|_\qq; \]
moreover, Lemmas \ref{lemma:Tsbounded} and Theorem \ref{thm: continuity of A} show that
\[I_2 +I_3  \le C \|T_\qq(A_\Gamma g)\|_\qq \le C\|A_\Gamma g\|_\qq \le C\|g\|_\qq.\]
This completes the proof.
\end{proof}

\begin{remark}\label{remark:DecompnotJohn}
    In Theorem~\ref{Decomp Thm} we only used that  $\Omega$ is a John domain to show that   $\frac{|W_t|}{|B_t|}\le C$ for any $t\in \Gamma$, $t\neq a$. If this estimate does not hold, then $A_\Gamma(g)$ would be replaced by the operator that maps $g$ to 
    \begin{equation*}
    \sum_{t\in \Gamma^*} \frac{|W_t|}{|B_t|} \frac{\chi_{B_t}(x)}{|W_t|}\int_{W_t} |g(y)| \dy.
    \end{equation*}
    For irregular domains, the factor $\frac{|W_t|}{|B_t|}$  becomes a \emph{weight} that can often be expressed in terms of a power of the distance to the boundary. In this context, the decomposition can be obtained provided that the operator $A_\Gamma$ is bounded in \emph{weighted spaces}. See \cite{L1,LGO1} for an  example in the classical Lebesgue spaces over H\"older-$\alpha$ domains.  {  We will consider this problem in a subsequent work.}
\end{remark}

%%%%%%%%%%%%%%%%%%%%%%%%%%%%%%%%%%%%%%%%%%%%%%%%%%%%%%%%%%%%%%%%%%%%%%%%%%%%%%%%%%%

\section{Improved Sobolev-Poincaré inequalities}
\label{section:improved-SPi}

In this section we prove Theorem~\ref{theorem:sp-general-intro} and related theorems.   To state these results, throughout this section
 we will assume that $\pp\in \Pp(\Omega)$ satisfies $1<p_-(\O)\le p_+(\O)<\infty$, and that $\qq\in \Pp(\Omega)$ is defined by \eqref{eq:def q-intro} and \eqref{eq:alpha-intro}.  For the convenience of the reader we repeat these definitions here:
\begin{equation}\tag{\ref{eq:def q-intro}}
\frac{1}{\pp}-  \frac{\alpha}{n} = \frac{1}{\qq},
\end{equation}
where $\alpha$ satisfies
\begin{equation}\tag{\ref{eq:alpha-intro}}
\begin{cases} 0\le \alpha < 1 &\textrm{ if } p_+(\O)<n \\
0\le \alpha< \frac{n}{p_+(\O)} & \textrm{ if } p_+(\O)\geq n. 
\end{cases}
\end{equation}
The restriction that $0\le \alpha<1$ is discused in Remark~\ref{rmk:alpha=1}.  

\medskip

To prove our main result on John domains using a local-to-global argument, we first need  to prove variable exponent Sobolev-Poincar\'e inequalities on cubes, which we will do using embedding theorems and the constant exponent Sobolev-Poincar\'e inequality.  We first recall the statement of this result; for a proof, see \cite[Lemma 2.1]{HH} or Bojarski~\cite{MR0982072}.  

\begin{lemma}\label{lemma:classical-sp}
Let $D\subset \R^n$ be a bounded John domain with parameter $\lambda$ from Definition \ref{definition:john}.  Fix $1\leq p<\infty$.  If $p<n$ and $p\le q\le p^*$, then for every $f\in W^{1,p}(D)$,
\begin{equation}\label{eq:classical-sp-case1}
\|f-f_D\|_{L^q(D)}
\le C(n,p,\lambda)|D|^{\frac{1}{n}+\frac{1}{q}-\frac{1}{p}}\|\nabla f\|_{L^p(D)}.
\end{equation}
If $p\ge n$ and $q<\infty$, then for every $f\in W^{1,p}(D)$,
\begin{equation}\label{eq:classical-sp-case2}
\|f-f_D\|_{L^q(D)}
\le C(n,q,\lambda)|D|^{\frac{1}{n}+\frac{1}{q}-\frac{1}{p}}\|\nabla f\|_{L^p(D)}.
\end{equation}
\end{lemma}

For our proof below we need better control over the constants in~\eqref{eq:classical-sp-case1} and~\eqref{eq:classical-sp-case2}, since in the variable exponent setting these constants will depend on the local values of $\pp$ and $\qq$.  The critical situation is when $p_-<n<p_+$, since in that case there may be cubes $U_t$ on which $p_-<n$ is very close to $n$, and the constant in \eqref{eq:classical-sp-case1} tends to infinity as $p$ approaches $n$ and $q$ approaches $p^*$.  There are similar concerns if $p_->n$ but is close to $n$.
In the latter case, we can choose a uniform constant in~\eqref{eq:classical-sp-case2}, as the following remark shows.

\begin{remark}\label{remark:constant-qbar}
In the case $p\ge n$ the constant can be taken depending only on $\bar{q}$ for any $\bar{q}\ge q$. Indeed, applying the H\"older inequality and \eqref{eq:classical-sp-case2} we obtain
\begin{align*}
\left( \int_Q |f(x)-f_Q|^q \dx\right)^\frac{1}{q} &\le |Q|^{\frac{1}{q}-\frac{1}{\bar{q}}} \left( \int_Q |f(x)-f_Q|^{\bar{q}}\right)^\frac{1}{\bar{q}}\\
&\le C(n,\bar{q})|Q|^{\frac{1}{q}-\frac{1}{\bar{q}}}|Q|^{\frac{1}{n}+\frac{1}{\bar{q}}-\frac{1}{p}} \left(\int_Q |\nabla f(x)|^p\dx\right)^\frac{1}{p} \\
&= C(n,\bar{q}) |Q|^{\frac{1}{n}+\frac{1}{q}-\frac{1}{p}} \left(\int_Q |\nabla f(x)|^p \dx\right)^\frac{1}{p}.
\end{align*}
\end{remark}

In the critical case when $p<n$, we can give a quantitative estimate on the constant in~\eqref{eq:classical-sp-case1}.  The proof of the following result is implicit in~\cite{Hajlasz2001}; here we give the details to make explicit the resulting constants.

\begin{lemma}\label{lemma:constant-sp-case1}
Let $Q\subset \R^n$ be a cube. If $1\le p<n$ and $p\le q\le p^*$, then
\begin{equation} \label{eqn:case1a}
|f-f_Q\|_{L^q(Q)} \le C(n)q |Q|^{\frac{1}{n}+\frac{1}{q}-\frac{1}{p}}\|\nabla f\|_{L^p(Q)}.
\end{equation}
\end{lemma}
\begin{proof}
First observe that by H\"older's inequality we have that
\[\|f-f_Q\|_{L^q(Q)} \le 2 \inf_{b\in \R} \left(\int_Q |f(x)-b|^q\dx\right)^\frac{1}{q}; \]
hence, it is enough to prove \eqref{eqn:case1a} with a constant $b$ in place of $f_Q$. 

We will take  $b$ to be the median of $f$ on $Q$, i.e., a possibly non-unique $b$ such that
\begin{equation}\label{eq:b}
    |\{x\in Q:\, f(x)\ge b\}|\ge |Q|/2\quad \text{ and } \quad |\{x\in Q:\, f(x)\le b\}|\ge |Q|/2.
\end{equation}

By~\cite[Corollary 6]{Hajlasz2001},  the inequality 
\begin{equation} \label{eqn:case1b}
\|f-b\|_{L^q(Q)} \le C_2 \|\nabla f\|_{L^p(Q)}
\end{equation}
is equivalent to 
\[\sup_{t>0} |\{x\in\O:\,|f(x) - b| > t\}|t^q \le C_1 \left(\int_\Omega |\nabla f(x)|^p \dx\right)^\frac{q}{p}. \]
Moreover, if we define
$$I_1^\O g(x) = \int_\O g(z)|x-z|^{1-n}\dz,$$
then by \cite[Theorem 10]{Hajlasz2001} we have that
$|g(x) -g_\Omega| \le C I_1^\O(\nabla g)(x)$. Finally, by \cite[Lemma~11]{Hajlasz2001},
\[\sup_{t>0}|\{x\in\O:\,I_1^\O g(x)>t\}|t^{n/(n-1)} \le C(n)\left(\int_\O |g(x)|\dx\right)^{n/(n-1)}.\]
If we combine these three estimates, we get
\begin{equation} \label{eqn:case1c}
\left(\int_Q |f(x)-b|^\frac{n}{n-1}\dx\right)^\frac{n-1}{n} \le C(n) \int_Q |\nabla f(x)|\dx,
\end{equation}
%}
This is inequality~\eqref{eqn:case1b} when $p=1$ and $q=1^*$. 

In \cite[Theorem 8]{Hajlasz2001}  inequality~\eqref{eqn:case1b} for $1<p<n$ and $q=p^*$ is derived from the case $p=1$. We reproduce their argument to keep track of the constant. Fix $b\in\R$ as in \eqref{eq:b} and let $Q_+=\{x\in Q:\, f(x)\ge b\}$ and $Q_-=\{x\in Q:\, f(x)\le b\}$.  Let $\gamma=p(n-1)/(n-p)$ and define
\[g(x) = \begin{cases}
                    |f(x)-b|^\gamma,  & x\in Q_+,\\
                    -|f(x)-b|^\gamma, & x\in Q_-.
                \end{cases} \]
Then $|g|^\frac{n}{n-1} = |f-b|^{p^*}$, and $g$ satisfies \eqref{eq:b} with $b=0$. Therefore,  inequality~\eqref{eqn:case1c} applied to $g$ and H\"older's inequality, we get
\begin{align*}
\left(\int_Q |f(x)-b|^{p^*} \dx\right)^\frac{n-1}{n} &= \left(\int_Q|g(x)|^\frac{n}{n-1} \dx\right)^\frac{n-1}{n} \\
&\le C(n) \int_Q |\nabla g(x)|\dx \\
&\le C(n) \int_Q \gamma |f(x)-b|^{\gamma-1}|\nabla f(x)| \dx \\
&\le C(n) \gamma \int_Q |f(x)-b|^{\frac{n(p-1)}{n-p}}|\nabla f(x)|\dx \\
&\le C(n)\frac{p(n-1)}{n-p} \left(\int_Q |f(x)-b|^{p^*} \dx\right)^{\frac{1}{p'}}\left(\int_Q |\nabla f(x)|^p\right)^\frac{1}{p}.
\end{align*}
If we rearrange terms, we get
\begin{equation} \label{eqn:case1d}
    \|f-b\|_{L^{p^*}(Q)} \le C(n)\frac{p(n-1)}{n-p} \|\nabla f\|_{L^p(Q)} \leq C(n)p^* \|\nabla f\|_{L^p(Q)}.
\end{equation}

Finally, to prove inequality~\eqref{eqn:case1a} with $q<p^*$ we follow the argument in \cite[Lemma~2.1]{HH}. Fix  $s$ such that $s^*=q$, (or $s=1$ if $q<1^*$); then by inequality~\eqref{eqn:case1d} with exponents $s^*$ and $s$ and  H\"older's inequality,
\begin{multline*}
\left(\int_Q |f(x)-f_Q|^q\dx\right)^\frac{1}{q} 
\le C(n) \frac{s(n-1)}{n-s} \left(\int_Q |\nabla f(x)|^s\dx\right)^\frac{1}{s} \\
\le C(n) \frac{s(n-1)}{n-s} \left(\int_Q |\nabla f(x)|^p\dx\right)^\frac{1}{p} |Q|^{\frac{1}{s}-\frac{1}{p}} 
\le C(n)s^*|Q|^{\frac{1}{n}+\frac{1}{q}-\frac{1}{p}} \|\nabla f\|_{L^p(Q)}.
\end{multline*}
This completes the proof.
\end{proof}

The following result was proved in \cite[Theorem~2.2]{HH} when $\qq=\pp$ and our proof is adapted from theirs.  This result  is the key for translating the classical Sobolev-Poincaré inequality to a variable exponent setting, with minimal hypotheses on the exponents. 

\begin{lemma}\label{lemma:local-sp}
Given a cube  $Q\subset\R^n$, let $\pp \in\Pp(Q)$ and suppose $1\leq p_-(Q)\leq p_+(Q) < \infty$.  Define $\qq$ by \eqref{eq:def q-intro}. If either $p_-(Q)<n$ and $p_-(Q)\le q_+(Q)\le (p_-(Q))^*$, or $p_-(Q)\ge n$ and $q_+(Q)<\infty$, then 
\begin{equation}\label{eq:local-sp}
\|f-f_Q\|_{L^\qq(Q)}
\le C(n,q_+(Q))(1+|Q|)^2 |Q|^{\frac{1}{n}+\frac{1}{q_+(Q)}-\frac{1}{p_-(Q)}}
\|\nabla f\|_{\Lp(Q)},
\end{equation}
Moreover, if $p_-(Q)<n$, then $C(n,q_+(Q)) = C(n)q_+(Q)$. 
\end{lemma}
\begin{proof}
We first consider the case  $p_-(Q)<n$.  By Lemma~\ref{lemma:pp-qq-imbed} (applied twice)  and by Lemma~\ref{lemma:constant-sp-case1}, we have that
\begin{align*}
    \|f-f_Q\|_\qq &\le (1+|Q|)\|f-f_Q\|_{q_+(Q)}\le C(n)q_+(Q)(1+|Q|)|Q|^{\frac{1}{n}+\frac{1}{q_+(Q)}-\frac{1}{p_-(Q)}}\|\nabla f \|_{p_-(Q)} \\
    &\le C(n)q_+(Q)(1+|Q|)^2 |Q|^{\frac{1}{n}+\frac{1}{q_+(Q)}-\frac{1}{p_-(Q)}}\|\nabla f\|_\pp.
\end{align*}
The proof for $p_-(Q)\geq n$ is the same, but using  \eqref{eq:classical-sp-case2} instead of Lemma~\ref{lemma:constant-sp-case1}.
\end{proof}

Our goal is to apply this inequality to the cubes of a tree-covering of a larger John domain. We can do so directly when the oscillation of $\pp$ on a cube is under control.  

\begin{lemma}\label{lemma:local-sp-direct}
Let $\Omega\subset \R^n$ be a bounded John domain with a tree covering $\{U_t\}_{t\in\Gamma}$. Given $\pp\in \Pp(\O)$, suppose $1\leq p_-(\Omega)\leq p_+(\Omega)<\infty$ and $\pp\in\partial LH_0^{\tau_K}(\O)$.  Define $\qq\in \Pp(\Omega)$ by \eqref{eq:def q-intro}. If $U\in\{U_t\}_{t\in\Gamma}$ is such that $p_-(U)<n$ and $p_-(U)\le q_+(U)\le p_-(U)^*$ or $p_-(U)\ge n$, then
\begin{equation}\label{eq:local-sp-direct}
\|f-f_U\|_{L^\qq(U)} 
\le C(\Omega,n,C_0,q_+(U)) |U|^{\frac{1-\alpha}{n}}\|\nabla f\|_{\Lp(U)}. 
\end{equation}
The constant $C_0$ is from Definition~\ref{tauLHB} and depends on $\pp$. In the first case, the constant can be taken  to be  $C(\Omega,n, C_0)q_+(U)$.
\end{lemma}

\begin{proof}
This is an immediate consequence of Lemma \ref{lemma:local-sp}. By inequality~\eqref{eq:local-sp} and Lemma \ref{lemma:LHBequiv}, and since  $|Q|\le |\Omega|$, we have that
\begin{align*}
    \|f-f_U\|_\qq 
    & \le C(n,q_+(U))(1+|Q|)^2 |U|^{\frac{1}{n}+\frac{1}{q_+(U)}-\frac{1}{p_-(U)}}\|\nabla f\|_\pp \\
    &= C(n,q_+(U)(1+|Q|)^2 |U|^{\frac{1-\alpha}{n}+\frac{1}{p_+(U)}-\frac{1}{p_-(U)}}\|\nabla f\|_\pp \\
    & \leq C(\Omega,n,C_0,q_+(U))|U|^{\frac{1-\alpha}{n}}\|\nabla f\|_\pp.
\end{align*}
\end{proof}

We do not want to assume, however, that the restriction on the oscillation of $\pp$ in Lemma~\ref{lemma:local-sp-direct} holds.  Indeed, we want to consider exponents $\pp$ such that $p_-(U)<n$ but $q_+(U)>p_-(U)^*$. To handle this situation, we prove an extension result, which shows that a  Sobolev-Poincaré inequality can be obtained for a domain $U$ if it holds for every set in a finite partition of $U$. This result is similar to~\cite[Theorem 2.6]{HH}; however, rather than work in the full generality of this theorem, we concentrate on  the particular case where $U$ is an element of a tree-decomposition of a larger John domain $\Omega$.  This allows us to obtain a sharper estimate on the constant. 

\begin{lemma}\label{lemma:sp-cube-extension}
Let $\O\subset\R^n$ a bounded John domain with a tree-covering $\{U_t\}_{t\in\Gamma}$.  Suppose $\pp\in\Pp(\O)$ satisfies $1\leq p_-(\Omega) \leq p_+(\Omega)<\infty$ and $\pp\in\partial LH_0^{\tau_K}(\O)$. Define  $\qq\in \Pp(\Omega)$ by~\eqref{eq:def q-intro}. For a fixed cube $U\in \{U_t\}_{t\in\Gamma}$, suppose that there exist cubes $G_i$ $i=1,\dots,M$, such that $G_i\subset U$ for every $i$, $U=\cup_{i=1}^M G_i$ and either $p_-(G_i)<n$ and  $p_-(G_i)\le q_+(G_i)\le (p_-(G_i))^*$, or $p_-(G_i)\ge n$. Thus Lemma \ref{lemma:local-sp} holds on each $G_i$; let $C_i$ be the constant for this inequality on $G_i$. Then there exists a constant $C>0$ such that for every $f\in W^{1,\pp}(U)$,
\[\|f-f_{U}\|_{L^\qq(U)} 
\le C|U|^{\frac{1-\alpha}{n}}\|\nabla f\|_{\Lp(U)}. \]
The constant $C$ depends on $n$, $M$, the ratio $|U|/|G_i|$, and $\max_i\{C_i\}$. 
\end{lemma}

\begin{proof}
We first apply the triangle inequality to get
\begin{align*}
    \|(f-f_U)\chi_U\|_\qq &\le \sum_{i=1}^M \|\chi_{G_i}(f-f_U)\|_\qq \\
    &\le \sum_{i=1}^M \|\chi_{G_i}(f-f_{G_i})\|_\qq + \sum_{i=1}^M \|\chi_{G_i}(f_U-f_{G_i})\|_\qq.
\end{align*}   
To estimate the first term,  we apply Lemma~\ref{lemma:local-sp} and  use the fact that $\frac{1}{n}+\frac{1}{q_+(G_i)}-\frac{1}{p_-(G_i)}\ge 0$, $U$ is bounded, and that $\frac{1}{q_+(U)}-\frac{1}{p_-(U)}\le \frac{1}{q_+(G_i)}-\frac{1}{p_-(G_i)}$:
\begin{align*}
\|\chi_{G_i}(f-f_{G_i})\|_\qq &\le C_i|G_i|^{\frac{1}{n}+\frac{1}{q_+(G_i)}-\frac{1}{p_-(G_i)}} \|\chi_{G_i}\nabla f\|_\pp \\
&\le C_i|U|^{\frac{1}{n}+\frac{1}{q_+(G_i)}-\frac{1}{p_-(G_i)}} \|\nabla f\|_\pp \\
&\le C_i|U|^{\frac{1}{n}+\frac{1}{q_+(U)}-\frac{1}{p_-(U)}} \|\nabla f\|_\pp;
\intertext{by \eqref{eq:def q-intro} and Lemma \ref{lemma:LHBequiv},}
     &\le C_i|U|^{\frac{1-\alpha}{n} +\frac{1}{p_+(U_t)}-\frac{1}{p_-(U_t)}} \|\nabla f\|_\pp\\
    &\le C_i|U|^{\frac{1-\alpha}{n}} \|\nabla f\|_\pp.
\end{align*}

To estimate the second term, we first note that
\[\|\chi_{G_i}(f_U-f_{G_i})\|_\qq = \|\chi_{G_i}\|_\qq |f_U-f_{G_i}| \le \|\chi_U\|_\qq |f_U-f_{G_i}|.\]
By the classical Poincaré inequality in $L^1(U)$ (see, for example,~\cite{AcostaDuran}), we have that
\begin{align*}
|f_U-f_{G_i}| 
& \le |G_i|^{-1}\int_{G_i} |f(x)-f_U|\dx \\
& \le |G_i|^{-1}\int_{U} |f(x)-f_U|\dx   \\
&\le C|G_i|^{-1}\textrm{diam}(U)\int_U |\nabla f(x)|\dx.
\intertext{By \eqref{eqn:holder} and since $\diam(U)\sim |U|^{\frac{1}{n}}$, we have} 
 &\le C|G_i|^{-1} |U|^\frac{1}{n} \|\chi_U\|_\cpp \|\nabla f\|_\pp.
\end{align*}
If we combine this with the previous estimate and apply inequality \eqref{eqn:offdiag-K0-1}, we get
\begin{align*}
\|\chi_{G_i}(f_U-f_{G_i})\|_\qq &\le C|G_i|^{-1}|U|^\frac{1}{n} \|\chi_U\|_\qq \|\chi_U\|_\cpp \|\nabla f\|_\pp \\
&\le C|G_i|^{-1}|U|^\frac{1}{n}|U|^{1-\frac{\alpha}{n}} \|\nabla f\|_\pp\\
&= C\frac{|U|}{|G_i|}|U|^{\frac{1-\alpha}{n}}\|\nabla f\|_\pp.
\end{align*}
This completes the proof.  
\end{proof}

In order to apply Lemma~\ref{lemma:sp-cube-extension},  we need to partition each $U_t$ into smaller sets $G_i$ such that one of the conditions in Lemma \ref{lemma:local-sp} holds on $G_i$. To do so, we will assume that the local oscillation of $\pp$ is under control.  In particular, we will assume that $\frac{1}{\pp}$ is $\frac{\sigma}{n}$ continuous for some $\sigma<1-\alpha$. 

\begin{lemma}\label{lemma:Ut-sp}
Let $\O\subset\R^n$ a bounded John domain with a tree-covering $\{U_t\}_{t\in\Gamma}$. Suppose $\pp\in\Pp(\O)$ satisfies $1<p_-(\O)\le p_+(\O)<\infty$ and $\pp \in \partial LH_0^{\tau_K}(\Omega)$.
Define $\qq\in\Pp(\Omega)$ by \eqref{eq:def q-intro}. If $\frac{1}{\pp}$ is uniformly $\frac{\sigma}{n}$-continuous for some $0<\sigma< 1-\alpha$, then there is a constant $C$, independent of $t$, such that
\begin{equation}\label{eq:local-lemma}
    \|f-f_{U_t}\|_{L^\qq(U_t)} 
    \le C |U_t|^{\frac{1-\alpha}{n}} \|\nabla f\|_{\Lp(U_t)},
\end{equation}
for every  $f\in W^{1,\pp}(U_t)$ and every $t\in\Gamma$. The constant $C$ depends on $n$, $p_+(\Omega)$, $\alpha$ and $\sigma$. In particular, it goes to infinity when $\alpha$ tends to $1$ or  when $\sigma$ tends to $1-\alpha$. 
\end{lemma}

\begin{proof}
Fix $\delta$ in the $\frac{\sigma}{n}$-continuity condition for $\frac{1}{\pp}$.  
We first consider the cubes $U_t$ such that diam$(U_t)<\delta$. If $p_-(U_t)\ge n$, then by Lemma~\ref{lemma:local-sp-direct} and Remark~\ref{remark:constant-qbar}, we have that inequality~\eqref{eq:local-lemma} holds  with a constant that depends only on $\Omega$, $n$ and $q_+(\Omega)$. On the other hand, if $p_-(U_t)<n$, then $U_t$ is contained in the ball $B(x_t,\delta/2)$, where $x_t$ is the center of $U_t$.  Therefore, given any two points $x,\,y \in U_t$, $y \in B(x,\delta)$, and so
\[ \left| \frac{1}{p(x)} - \frac{1}{p(y)}\right|  < \frac{\sigma}{n}.  \]
If we now argue as in the proof of Lemma~\ref{lemma:LHBequiv}, we get that 
\begin{equation} \label{eqn:osc}
    \frac{1}{p_-(U_t)}-\frac{1}{q_+(U_t)} = \frac{1}{p_-(U_t)}-\frac{1}{p_+(U_t)} + \frac{\alpha}{n} \le \frac{\sigma+\alpha}{n} < \frac{1}{n};
\end{equation}
hence, $p_-(U_t)\le q_+(U_t)< (p_-(U_t))^*$. Again, Lemma \ref{lemma:local-sp-direct} gives  inequality\eqref{eq:local-lemma}, but now with a constant of the form $C(n)q_+(U_t)$.  We need to show that $q_+(U_t)$ is uniformly bounded. But by inequality~\eqref{eqn:osc},
\[\frac{1}{q_+(U_t)}\ge \frac{1}{p_-(U_t)}-\frac{\sigma+\alpha}{n} = \frac{n-(\sigma+\alpha)p_-(U_t)}{np_-(U_t)}.\]
Since we have assumed that $p_-(U_t)<n$, we get that
\[q_+(U_t) \le \frac{n}{1-\sigma-\alpha},\]
which gives the desired upper bound.
Note that in this case, the constant may blow up if $\sigma$ tends to $1-\alpha$ or if $\alpha$ tends to $1$. 

\medskip

Now suppose $\diam(U_t)\geq \delta$; for this case we will apply Lemma~\ref{lemma:sp-cube-extension}.  
Partition $U_t$ into $M_t$ cubes $\{G_i\}_{i=1}^{M_t}$ with $\frac{\delta}{2}<\diam(G_i)<\delta$.  The number of cubes is uniformly bounded, since
\[ M_t \lesssim \frac{|U_t|}{\delta^n} \leq \frac{|\Omega|}{\delta^n}. \]

If $p_-(G_i)\ge n$, then $G_i$ satisfies the hypotheses of Lemma \ref{lemma:sp-cube-extension}. If, $p_-(G_i)<n$, we can repeat the previous oscillation estimate for $G_i$ instead of $U_t$, to get that  $p_-(G_i)\le q_+(G_i)\le (p_-(G_i))^*$.  So  again $G_i$ satisfies the hypotheses of Lemma \ref{lemma:sp-cube-extension}. Therefore, we can apply this lemma to $U_t$; this gives us~\eqref{eq:local-lemma}with a constant that depends on the number of cubes $M_t$, which is bounded, and on the ratio $|U_t|/|G_i|$, which is also bounded by $|\Omega|/\delta^n$.  The constant also depends on the largest constant $C_i$ for the Sobolev-Poincar\'e inequality on $G_i$. To estimate these  we can argue as we did for the cubes satisfying $\diam(U_t)<\delta$  to get that they are uniformly bounded. This gives the desired result.  
\end{proof}

The following two lemmas are corollaries of the proof of Lemma~\ref{lemma:Ut-sp}; we show that with additional restrictions on the oscillation of $\pp$ the hypotheses can be weakened.

\begin{lemma}\label{lemma:local-sp-case1}
Let $\Omega$, $\pp$,  and $\qq$ be as in Lemma~\ref{lemma:Ut-sp}, with the additional restriction $p_+(\Omega)<n$ but only assuming that $\frac{1}{\pp}$ is $\frac{1-\alpha}{n}$-continuous. Then \eqref{eq:local-lemma} holds with a constant that depends on $n$, $\Omega$ and $p_+(\Omega)$, but not on $\alpha$ (or equivalently, on $\qq$).
\end{lemma}

\begin{proof}
Since $p_-(U_t)<n$ for every $t\in\Gamma$, we can apply Lemma~\ref{lemma:local-sp-direct} on every $U_t$; this yields a local constant of the form $C(n)q_+(U_t)$. Arguing as in the proof of  Lemma~\ref{lemma:Ut-sp} with our weaker continuity assumption,  we have $q_+(U_t)\le p_-(U_t)^* \le p_+(\Omega)^*<\infty$, and so the local constants are uniformly bounded. 
\end{proof}

\begin{lemma}\label{lemma:local-sp-case2}
Let $\Omega$, $\pp$, and $\qq$ be as in Lemma~\ref{lemma:Ut-sp}, with the additional restriction $p_-(\Omega)\ge n$, but not assuming $\frac{1}{\pp}$ is $\varepsilon$-continuous for any $\varepsilon>0$. Then \eqref{eq:local-lemma} holds with a constant that depends on $n$, $\Omega$ and $q_+(\Omega)$.
\end{lemma}

\begin{proof}
Since $p_-(U_t)\ge n$ for every $t\in\Gamma$ we can again apply Lemma \ref{lemma:local-sp-direct} on every $U_t$, without assuming any additional regularity on $\pp$. The local constants are of the form $C(n,q_+(U_t))$, but by Remark~\ref{remark:constant-qbar} we can  replace $q_+(U_t)$ by $q_+(\Omega)$. 
\end{proof}

\begin{remark}\label{rmk:alpha=1}
In our definition of $\alpha$ in~\eqref{eq:alpha-intro} above, we required that $0\le\alpha<1$ when $p_+(\Omega)<n$.  Ideally, we would like to take $\alpha=1$ in  Lemmas~\ref{lemma:Ut-sp}--\ref{lemma:local-sp-case2}; however, this is not possible unless $\pp$ is constant.   This is a consequence of the fact that our proof uses Lemma \ref{lemma:local-sp}. Indeed, if we fix a cube $Q$ and take $\alpha=1$, we have that 
\begin{equation*}
    0 \le \frac{1}{p_-(Q)}-\frac{1}{p_+(Q)} 
    = \frac{1}{p_-(Q)}-\frac{1}{q_+(Q)}-\frac{1}{n} 
    = \frac{1}{p_-(Q)^*}-\frac{1}{q_+(Q)}.
\end{equation*}
In other words, if we assume that  $q_+(Q)\le p_-(Q)^*$, then this inequality forces $p_+(Q)=p_-(Q)$, so that $\pp$ is constant.
The advantage of Lemma \ref{lemma:local-sp} is that it imposes minimal regularity conditions on $\pp$. 

\begin{comment}
The  Sobolev-Poincar\'e inequality is true in the critical case $\qq = p^*(\cdot)$ if we assume $1/\pp$ is log-H\"older continuous  and $p_+(\Omega)<n$: see, for example, \cite[Theorem 6.29]{CruzUribeFiorenza} and \cite[Theorem 8.3.1]{DieningBook}. An interesting open question is whether  this inequality is true on cubes assuming, for example, that $\pp$ is  continuous. If  that were the case, then the proof of Lemma \ref{lemma:Ut-sp}  can still work assuming continuity of $\pp$, and Theorem~\ref{theorem:sp-general-intro}  below would still hold.   This question is related to a similar open question for the Sobolev inequality:  see~\cite[Problem~A.22]{CruzUribeFiorenza}.
\end{comment}
\end{remark}

We can now prove our main result;  for the convenience of the reader we repeat its statement.

\begin{theorem*} \ref{theorem:sp-general-intro}:
Let $\O\subset \R^n$ be a bounded John domain.  Suppose  $\pp\in\Pp(\Omega)$ is such that $1<p_-(\O)\le p_+(\O)<\infty$ and $\pp\in\partial LH_0^{\tau_K}(\Omega)$.  Fix $\alpha$ as in~\eqref{eq:alpha-intro} and define $\qq$ by~\eqref{eq:def q-intro}.  Suppose also that $\frac{1}{\pp}$ is uniformly $\frac{\sigma}{n}$-continuous for some $\sigma<1-\alpha$.  Then there is a constant $C$ such that for every $f\in W^{1,p(\cdot)}(\Omega)$,
\begin{equation}\label{eq:sp-john}
\|f(x)-f_\Omega\|_{q(\cdot)} \le C \|d^{1-\alpha}\nabla f(x)\|_{p(\cdot)},
\end{equation}
where $d(x)=\dist(x,\partial \Omega)$.
\end{theorem*}

\begin{remark}
We want to stress that the hypotheses on $\pp$ in Theorem~\ref{theorem:sp-general-intro} allow the exponent to be discontinuous. In fact, it can be discontinuous at every point of the domain, as long  as the jumps of $1/\pp$ are smaller than $\frac{\sigma}{n}$ and the oscillation of $\pp$ decays towards the boundary according to condition $\partial LH_0^{\tau_K}(\O)$.
\end{remark}

\begin{proof}
We will prove this result using  a local-to-global argument and the local Sobolev-Poincar\'e inequality in Lemma~\ref{lemma:Ut-sp}. Let $\{U_t\}_{t\in \Gamma}$ be a tree covering of $\Omega$ as in Proposition~\ref{prop:tcov-John}.   By~\eqref{eqn:assoc-norm}, there exists  $g\in L^{\cqq}(\O)$, $\|g\|_\cqq\le 1$, such that  
\begin{align*}
\|f-f_\O\|_{q(\cdot)} 
&\le C \int_\Omega (f(x)-f_\O)g(x) \dx \\
& = C \int_\Omega (f(x)-f_\O)(g(x)-g_\Omega) \dx; \\
\intertext{if we now decompose $g-g_\O$ into $\{g_t\}_{t\in\Gamma}$ using Theorem~\ref{Decomp Thm}, we get, since $\int_{U_t} g_t\,dx=0$, that}
&= C\sum_{t\in\Gamma} \int_{U_t} (f(x)-f_{\O})g_t(x) \dx\\ 
&=  C\sum_{t\in\Gamma} \int_{U_t} (f(x)-f_{U_t})g_t(x) \dx;\\ 
\intertext{we now apply~\eqref{eqn:holder} and Lemmas~\ref{lemma:Ut-sp} and~\ref{lemma:offdiag-Höldersum} with $\beta=\alpha$ to get }
&\le C \sum_{t\in\Gamma} C  \|\chi_{U_t} (f-f_{U_t})\|_\qq \|\chi_{U_t}g_t\|_\cqq\\
&\le 
C \sum_{t\in\Gamma} |U_t|^{\frac{1-\alpha}{n}} \|\chi_{U_t} \nabla f\|_\pp \left\| \chi_{U_t} \sum_{r\in \Gamma} \chi_{U_r} \frac{\|\chi_{U_r}g_r\|_\cqq}{\|\chi_{U_r}\|_\cqq}\right\|_\cqq  \\
&= C \sum_{t\in\Gamma} \ell(U_t)^{1-\alpha} \|\chi_{U_t} \nabla f\|_\pp \left\| \chi_{U_t} \sum_{r\in \Gamma} \chi_{U_r} \frac{\|\chi_{U_r}g_r\|_\cqq}{\|\chi_{U_r}\|_\cqq}\right\|_\cqq \\
&\le C \sum_{t\in\Gamma} \|\chi_{U_t} d^{1-\alpha} \nabla f\|_\pp \left\| \chi_{U_t} \sum_{r\in \Gamma} \chi_{U_r} \frac{\|\chi_{U_r}g_r\|_\cqq}{\|\chi_{U_r}\|_\cqq}\right\|_\cqq \\
&\le C \| d^{1-\alpha}\nabla f \|_{p(\cdot)} \left\| \sum_{r\in \Gamma} \chi_{U_r} \frac{\|\chi_{U_r}g_r\|_{\cqq}}{\|\chi_{U_r}\|_{\cqq}}\right\|_{\cqq};\\
\intertext{finally, we apply the bound from Theorem \ref{Decomp Thm} to get}
&\le C  \| d^{1-\alpha}\nabla f \|_{p(\cdot)} \|g\|_{\cqq} \\
&\le C\|d^{1-\alpha}\nabla f  \|_{p(\cdot)}.
\end{align*}
\end{proof}

\begin{comment}
{  
\begin{theorem}\label{theorem:sp-p->n}
Let $\O\subset \R^n$ a bounded John domain, $\pp\in\Pp(\Omega)$ satisfying $n\le p_-(\O)\le p_+(\O)<\infty$ and $\pp\in\partial LH_0^{\tau_K}$.  Let also $\qq\in \Pp(\Omega)$ defined by \eqref{eqn:pq-assump2} for some $0<\beta<\frac{n}{p_-(\Omega)}$ Then, there is a constant $C$ such that the following Sobolev-Poincar\'e inequality holds
\[\|f(x)-f_\Omega\|_{q(\cdot)} \le C \|d^{1-\beta}\nabla f(x)\|_{p(\cdot)},\]
for every $f\in W^{1,p(\cdot)}(\Omega)$.
\end{theorem}
\begin{proof}
The proof is exactly the same than the one for Theorem \ref{theorem:sp-general-intro}. The only difference is that here we apply the local inequalities given by Lemma \ref{lemma:sp-p->n} instead of Lemma \ref{lemma:Ut-sp}.
\end{proof}
}
\end{comment}

Theorem~\ref{theorem:sp-general-intro} gives a family of improved Sobolev-Poincaré inequalities. We highlight the case when $\alpha=0$ as a separate corollary.  

\begin{corollary}\label{cor:imp-poincare}
Let $\O\subset \R^n$ a bounded John domain.  Suppose  $\pp\in\Pp(\Omega)$ is such that  $1<p_-(\O)\le p_+(\O)<\infty$, $\pp\in\partial LH_0^{\tau_K}$ and $\frac{1}{\pp}$ is uniformly $\frac{\sigma}{n}$-continuous for some $\sigma<1$. Then for every $f\in W^{1,p(\cdot)}(\Omega)$,
\[\|f(x)-f_\Omega\|_\pp \le C \|d\nabla f(x)\|_\pp,\]
where $d(x) = d(x,\partial\O)$.
\end{corollary}

If we impose further restrictions on the oscillation of $\pp$, we can weaken some of the other hypotheses in Theorem~\ref{theorem:sp-general-intro}.  

\begin{theorem} \label{thm:sp-small-pp}
Let $\O$, $\pp$, $\alpha$ and $\qq$ be as in Theorem \ref{theorem:sp-general-intro}, but also assume that $p_+(\O)<n$ and  $\frac{1}{\pp}$ is $\frac{1-\alpha}{n}$-continuous.  Then inequality $\eqref{eq:sp-john}$ holds. 
\end{theorem}

\begin{proof}
The proof is the same as the proof of Theorem \ref{theorem:sp-general-intro}, but we use Lemma~\ref{lemma:local-sp-case1} instead of Lemma~\ref{lemma:Ut-sp}.
\end{proof}

\begin{theorem} \label{thm:sp-big-pp}
Let $\O$, $\pp$, $\alpha$ and $\qq$ be as in Theorem \ref{theorem:sp-general-intro}, but also assume that $p_-(\O)\ge n$ and do not assume that $\frac{1}{\pp}$ is $\varepsilon$-continuous for any $\varepsilon>0$. Then inequality $\eqref{eq:sp-john}$ holds. 
\end{theorem}

\begin{proof}
Again, the proof is the same as the proof of Theorem~\ref{theorem:sp-general-intro}, but we use Lemma \ref{lemma:local-sp-case2} instead of Lemma~\ref{lemma:Ut-sp}. 
\end{proof}

%%%%%%%%%%%%%%%%%%%%%%%%

\section{The Sobolev inequality with rough exponents}
\label{section:sobolev}

In this section we prove Theorem~\ref{thm:sobolev-intro}.   For the convenience of the reader we restate it here.

\begin{theorem*}\ref{thm:sobolev-intro}:
Let $\O\subset \R^n$ be a bounded  domain.  Suppose  $\pp\in\Pp(\Omega)$ is such that $1<p_-(\O)\le p_+(\O)<\infty$.  Fix $\alpha$ as in~\eqref{eq:alpha-intro} and define $\qq$ by~\eqref{eq:def q-intro}.  Suppose also that $\frac{1}{\pp}$ is uniformly $\frac{\sigma}{n}$-continuous for some $\sigma<1-\alpha$.  Then there is a constant $C$ such that for every $f\in W_0^{1,p(\cdot)}(\Omega)$,
\begin{equation} \label{eqn:sobolev-ineq}
    \|f\|_{L^\qq(\Omega)} \le C \|\nabla f(x)\|_{L^\pp(\Omega)}.
\end{equation}
\end{theorem*}

\medskip

The heart of the proof is an extension theorem for exponent functions that are $\varepsilon$-continuous, which lets us extend the exponent $\pp$ from an arbitrary bounded domain to a John domain on which the hypotheses of Theorem~\ref{theorem:sp-general-intro} hold.  We prove this extension result in three lemmas.  The first lemma lets us extend the exponent to the boundary of $\Omega$.

\begin{lemma}  \label{lemma:uniform-boundary}
Given a bounded domain $\Omega$, let $f : \Omega \rightarrow \R$ be uniformly $\varepsilon_0$-continuous on $\Omega$ for some $\varepsilon_0>0$.  Then $f$ can be extended to a function on $\overline{\Omega}$ that is uniformly $\varepsilon$-continuous for any $\varepsilon>\varepsilon_0$.  Further, $f_-(\overline{\Omega})=f_-(\Omega)$ and $f_+(\overline{\Omega})=f_+(\Omega)$.
\end{lemma}

\begin{proof}
Let $\delta_0>0$ be such that if $x,\,y\in \Omega$ are such that $|x-y|<\delta_0$, then $|f(x)-f(y)|<\varepsilon_0$.  Since $\overline{\Omega}$ is closed and bounded, it can be covered by a finite collection of balls $B(x_i,\delta_0/3)$, $1\leq i \leq N$, where $x_i \in \overline{\Omega}$.  Suppose $x_i\in \overline{\Omega}\setminus \Omega$.  Then it is a limit point of $\Omega$, and so there exists a point $x_i' \in \Omega$ such that $|x_i-x_i'|< \delta_0/6$.  But then $B(x_i,\delta_0/3) \subset B(x_i',\delta_0/2)$.  Therefore, we may assume that $\overline{\Omega}$ is covered by $N$ balls $B(x_i,\delta_0/2)$, where each $x_i\in \Omega$.  

Fix a point $x\in \overline{\Omega}\setminus \Omega$.  Then $x\in B(x_i,\delta_0/2)$ and so there exists a sequence of points $y_k \in B(x_i,\delta_0/2)\cap \Omega$ such that $y_k \rightarrow x$ as $k \rightarrow \infty$.  Then we have that $|f(x_i)-f(y_k)|<\varepsilon_0$, so the sequence $\{f(y_k)\}$ is bounded.  Therefore, we can define
\[ f(x) = \liminf_{k\rightarrow \infty} f(y_k).  \]
Hereafter, by passing to a subsequence we will assume without loss of generality that $f(y_k)\rightarrow f(x)$.

\medskip

It is immediate from the construction that $f_-(\overline{\Omega})=f_-(\Omega)$ and $f_+(\overline{\Omega})=f_+(\Omega)$.  To show $\varepsilon$-continuity, fix $\varepsilon>\varepsilon_0$ and $\delta=\delta_0/2$.   First, fix $x\in \Omega$ and suppose $y\in \overline{\Omega}$ satisfies $|x-y|<\delta$.  If $y \in \Omega$, then it is immediate that $|f(x)-f(y)|<\varepsilon_0<\varepsilon$.  Now suppose $y\in \overline{\Omega}\setminus \Omega$.  Then for some $i$, $y \in B(x_i,\delta_0/2)$ and and there is a sequence $\{y_k\}$ in this ball such that $f(y_k)\rightarrow f(y)$.  But then, we have that 
\[ |f(x)-f(y)| \leq |f(x)-f(y_k)|+|f(y_k)-f(y)| < \varepsilon_0 + |f(y_k)-f(y)|.  \]
If we take the limit as $k\rightarrow \infty$, we get that $|f(x)-f(y)|\leq \varepsilon_0<\varepsilon$.  

Now suppose $x\in \overline{\Omega}\setminus \Omega$ and fix $y\in \overline{\Omega}\setminus \Omega$ such that $|x-y|<\delta$.  Then we have $x\in B(x_i,\delta_0/2)$ and sequence $\{y_k\}$ in the ball such that $f(y_k)\rightarrow f(x)$; we also have that $y\in B(x_j,\delta_0/2)$ and sequence $\{z_k\}$ in this ball such that $f(z_k) \rightarrow f(y)$.  But then we have that for $k$ sufficiently large
\[ |y_k-z_k| \leq |x-y_k| + |x-y| + |y-z_k| < \delta_0/4 + \delta_0/2+ \delta_0/4 = \delta_0, \]
and so,
\begin{align*}
 |f(x)-f(y)| 
& \leq |f(x)- f(y_k)| + |f(y_k) - f(z_k)|+|f(z_k)-f(y)| \\
& \leq  |f(x)- f(y_k)| + \varepsilon_0 + |f(z_k)-f(y)|.
\end{align*} 
If we take the limit as $k\rightarrow \infty$, we get  $|f(x)-f(y)|\leq \varepsilon_0<\varepsilon$.  

Finally, we need to consider the case $x\in \overline{\Omega}\setminus \Omega$ and $y \in \Omega$; this is similar to but simpler than the previous case and we omit the details.  This completes the proof.
\end{proof}

\begin{remark}
    In the proof of Lemma~\ref{lemma:uniform-boundary}, the  value of $f(x)$ that we chose is not unique, and we could choose any value between the limit infimum and the limit supremum, or we could choose any other sequence $\{y_k\}$.  
\end{remark}

\medskip

The second lemma lets us extend the exponent function to an $\varepsilon$-continuous function on $\R^n\setminus \overline{\Omega}$.  A more general result, replacing balls by other sets, is possible using the same ideas; here we prove what is necessary for our proof of Theorem~\ref{thm:sobolev-intro}.  Details of generalizations are left to the interested reader.

\begin{lemma} \label{lemma:closed-extension}
Let $\Omega$ be a bounded domain, and suppose that $f : \overline{\Omega} \rightarrow \R$ is uniformly $\varepsilon$-continuous for some $\varepsilon>0$.  Let $B$ be a ball containing $\overline{\Omega}$.  Then there exists an extension of $f : 3\overline{B} \rightarrow \R$ that is uniformly $\varepsilon$-continuous, continuous on $3\overline{B}\setminus \overline{\Omega}$, and $f_-(\overline{\Omega})=f_-(3\overline{B})$ and $f_+(\overline{\Omega})=f_+(3\overline{B})$.
\end{lemma}

\begin{proof}
We will follow the proof of the extension theorem in Stein~\cite[Chapter~VI, Section~2]{S_SingularIntegrals}.  For the convenience of the reader we will adopt the same notation.  Let $\{Q_k\}$ be the Whitney decomposition of $\R^n\setminus \overline{\Omega}$:  that is, cubes with disjoint interiors such that 
\[ \diam(Q_k) \leq \dist(Q_k,\overline{\Omega}) \leq 4\diam(Q_k).   \]
Let $Q_k^*=\frac{9}{8}Q_k$, and let $\{q_k^*\}$ be a $C^\infty$ partition of unity such that $\supp(q_k^*)\subset Q_k^*$.   These cubes are such that given any $x\in \R^n \setminus \overline{\Omega}$ there exist at most $N$ cubes such that $x\in Q_k^*$.  Since $\overline{\Omega}$ is compact, there exists $p_k \in \overline{\Omega}$ such that $\dist(Q_k,p_k)=\dist(Q_k,\overline{\Omega})$.  
For all $x\in 3\overline{B}\setminus \overline{\Omega}$, define
\[ f(x) = \sum_k f(p_k) q_k^*(x). \]
Since $\sum q_k^*(x)=1$, is immediate from this definition that $f_-(\overline{\Omega})=f_-(3\overline{B})$ and $f_+(\overline{\Omega})=f_+(3\overline{B})$.  Since the cubes $Q_k^*$ have finite overlap, this sum contains a finite number of non-zero terms at each $x$ and so is continuous on $3\overline{B}\setminus \overline{\Omega}$.

We will show that with this definition, $f$ is uniformly $\varepsilon$-continuous on $3\overline{B}$.  
Since $f$ is uniformly $\varepsilon$-continuous on $\overline{\Omega}$, fix $\delta>0$ such that if $x,\,y\in \overline{\Omega}$ satisfy $|x-y|<\delta$, then $|f(x)-f(y)|<\varepsilon$.   Let
\[ K_0=\{ x\in 3\overline{B} : \dist(x,\overline{\Omega}) \geq \delta/8\}, \quad
K=\{ x\in 3\overline{B} : \dist(x,\overline{\Omega}) \geq \delta/4\}\]
Since $K_0$ is compact and $f$ is continuous at each $x\in K$, $f$ is uniformly continuous on $K_0$.  Therefore, there exists $\delta_{ext}<\delta/8$ such that if $x\in K$, and $y \in 3\overline{B}$ is such that $|x-y|<\delta_{ext}$, then $|f(x)-f(y)|<\varepsilon$.

Now fix $x\in 3\overline{B}\setminus \overline{\Omega}$ such that $\dist(x,\overline{\Omega})<\delta/4$.  Fix $y\in\overline{\Omega}$ such that $|x-y|<\delta/4$.  Then, since $\sum q_k^*(x)=1$,
\[ |f(x)-f(y)| \leq \sum_k |f(p_k)-f(y)|q_k^*(x).  \]
Fix $k$ such that $x\in Q_k^*$.  Then
\begin{multline*} \dist(Q_k, p_k) = \dist(Q_k,\overline{\Omega})
\leq \dist(\overline{\Omega},Q_k^*) + \dist(Q_k^*,Q_k) \\
< \dist(\overline{\Omega}, x) + \tfrac{1}{8}\diam(Q_k)
\leq \delta/4 + \tfrac{1}{8}\dist(Q_k,p_k).  
\end{multline*}
Therefore, rearranging terms we see that $\dist(Q_k,p_k) < \frac{2}{7}\delta$.  
Furthermore, we have that 
\[ |x-p_k| 
\leq \dist(Q_k,p_k)+ \diam(Q_k^*)
= \dist(Q_k,p_k) + \tfrac{9}{8}\diam(Q_k)
\leq \tfrac{17}{8}\dist(Q_k,p_k) < \tfrac{17}{28}\delta. 
\]
Hence, $|y-p_k|\leq |x-y|+|x-p_k|< \tfrac{1}{4}\delta + \tfrac{17}{28}\delta< \delta$.  Therefore, $|f(p_k)-f(y)|<\varepsilon$ for every $k$, and so we have that $|f(x)-f(y)|<\varepsilon$.

Finally, we need to show that if $x \in \overline{\Omega}$, then $f$ is $\varepsilon$-continuous at $x$.  But this follows by the previous argument, exchanging the roles of $x$ and $y$.  Therefore, we have shown that $f$ is $\varepsilon$-continuous on $3\overline{B}$.
\end{proof}

The final lemma shows how the extension can be modified to satisfy the $\tauLH(\Omega)$ condition on the larger set.

\begin{lemma} \label{lemma:extend-exp}
Given a bounded domain $\Omega$ and $\pp \in \Pp(\Omega)$, suppose $1\leq p_-\leq p_+ <\infty$ and $\frac{1}{\pp}$ is uniformly $\varepsilon_0$-continuous for some $\varepsilon_0>0$.  Fix a ball $B$ such that $\overline{\Omega}\subset B$. Then there exists an extension of $\pp$ to $3\overline{B}$ such that $p_-(3\overline{B})=p_-(\Omega)$, $p_+(3\overline{B}) =p_+(\Omega)$, and $\frac{1}{\pp}$ is uniformly $\varepsilon$-continuous on $3B$ for any $\varepsilon>\frac{p_+(\Omega)^2}{p_-(\Omega)^2}\varepsilon_0$.  Moreover, for any $\tau>1$, $\pp \in \tauLH(3B)$. 
\end{lemma}

\begin{proof}
We will prove that if $\pp$ is uniformly $\varepsilon_0$-continuous on $\Omega$, then it has an extension to $3\overline{B}$ that is uniformly $\varepsilon$-continuous for any $\varepsilon>\varepsilon_0$ and satisfies $p_-(3\overline{B})=p_-(\Omega)$, $p_+(3\overline{B}) =p_+(\Omega)$. The continuity estimate for $\frac{1}{\pp}$ follows immediately from this and the fact that for any $x,\,y$ in the domain of $\pp$,
\[ | p(x)-p(y)| \leq p_+^2\bigg|\frac{1}{p(x)}-\frac{1}{p(y)}\bigg|, \qquad
\bigg|\frac{1}{p(x)}-\frac{1}{p(y)}\bigg| \leq \frac{1}{p_-^2}| p(x)-p(y)|.
\]

On $\Omega$, define $f(x)=p(x)-p_-$.  Then by Lemmas~\ref{lemma:uniform-boundary} and~\ref{lemma:closed-extension}, for any $\varepsilon>\varepsilon_0$, we can extend $f$ to a uniformly $\varepsilon$-continuous function $3\overline{B}$ such that $f_-(3\overline{B}) = 0$, $f_+(3\overline{B})=p_+-p_-$, and $f$ is continuous on $3\overline{B}\setminus \overline{\Omega}$.  Let $\psi$ be a Lipschitz cut-off function such that $0\leq \psi(x)\leq 1$, $\psi(x)=1$ if $x\in 2B$, and $\psi(x)=0$ if $x\in \R^n\setminus 3B$.  Then, arguing as we did above, since $f\psi$ is uniformly continuous on $3\overline{B}\setminus 1.5B$, we must have that $f\psi$ is uniformly $\varepsilon$-continuous on $3\overline{B}\setminus 2B$, and so uniformly $\varepsilon$-continuous on $3\overline{B}$.  

Now define the extension of $\pp$ to be $p(x)=f(x)\psi+p_-$.  The $\varepsilon$-continuity and oscillation bounds for $\pp$ follow at once from the corresponding properties of $f$.  Finally, if we fix $x\in 3B\setminus 2B$ such that $\tau d(x) < \min\{ r(B),1/2\}$, then, since $\pp$ is Lipschitz on $B_{x,\tau}$, we have that \eqref{Boundary LH constant} holds.  On the other hand,  if $1/2>\tau d(x)>r(B)$, this condition always holds with a suitably large constant.  Thus, we have that $\pp \in \tauLH(3B)$.
\end{proof}

We can now prove our main result.

\begin{proof}[Proof of Theorem~\ref{thm:sobolev-intro}]
By Lemma~\ref{lemma:extend-exp}, we can extend $\pp$ to an exponent function on the ball $3B$ (which is a John domain) that satisfies all the hypotheses of Theorem~\ref{theorem:sp-general-intro}.  Therefore, we have that for all $f\in W_0^{1,\pp}(\Omega) \subset W^{1,p}(3B)$,
\[ \|f-f_{3B}\|_{L^\qq(3B)} \leq C\|d^{1-\alpha}\grad f\|_{L^\pp(3B)} \leq C\diam(\Omega)^{1-\alpha}\|\grad f\|_{L^\pp(3B)}.  
 \]
By the triangle inequality and since $\supp(f)\subset \Omega$, we can rewrite this as
\[ \|f\|_{L^\qq(\Omega)} \leq C \|\grad f\|_{L^\pp(\Omega)} + |f_{3B}|\|\chi_{3B}\|_{L^\qq(3B)}. \]
But by the classical Sobolev inequality in $L^1(\Omega)$ and by~\eqref{eqn:holder},
\[ |f_{3B}| \leq |3B|^{-1}\int_\Omega |f|\,dx \leq C|3B|^{-1}\int_\Omega |\grad f|\,dx
\leq C|3B|^{-1}\|\Omega\|_{L^\cpp(\Omega)}\|\grad f\|_{L^\pp(\Omega)}.  
\]
If we combine this with the previous inequality, we get the desired result.
\end{proof}

%%%%%%%%%%%%%%%%%%%%%%%%%%%%%%%%%%%%%%%%%%%%%%%%%%%%%%%%%%%%%%%%%%%%%%%%%%%%%%%%%%%

\section{Necessity of the log-H\"older continuity conditions}
\label{section:necessity}

In this section we consider the necessity of the $\tauLH(\Omega)$ and $LH_0(\Omega)$ conditions for Sobolev-Poincar\'e and Sobolev inequalities in the variable Lebesgue spaces.  We do not prove that they are necessary in general; rather, we construct examples to show that they are the weakest continuity conditions which can be assumed to prove these inequalities in general.  These results are similar to the classic example of  Pick and  R{\r{u}}{\v{z}}i{\v{c}}ka~\cite{MR1876258}  showing that log-H\"older continuity is necessary for the Hardy-Littlewood maximal operator to be bounded on $\Lp(\Omega)$.  Indeed, our construction should be compared to theirs. It also contains ideas similar to the ones used in \cite{ADLg}, where the irregularity appears on the domain instead of the exponent. 

We first show that the boundary condition $\tauLH(\Omega)$ is necessary for the validity of the improved Sobolev-Poincar\'e inequality \eqref{eq:sp-john},
\[\|f-f_\Omega\|_{L^\qq(\Omega)} \le C \|d^{1-\alpha}\nabla f\|_{L^\pp(\Omega)}, \]
for all $0 \leq \alpha \leq 1$. 
Indeed, given a bounded domain $\Omega\subset \R^n$, we construct an exponent function $\pp$ that is uniformly continuous on $\Omega$ but does not verify $\tauLH(\Omega)$ for any $\tau\geq 1$. Then, we construct  an associated sequence of  test functions for which  inequality~\eqref{eq:sp-john} holds with a constant that blows up.

Second, we  use the same uniformly continuous exponent $\pp$ (with some minor modifications) to show that the boundary condition $\tauLH(\Omega)$ is not sufficient for the validity of the Sobolev-Poincar\'e inequality obtained for $\alpha=1$ and $\qq=p^*(\cdot)$ in \eqref{eq:sp-john}, that is,
\begin{equation} \label{eqn:ps-endpt}
\|f-f_\Omega\|_{L^\ps(\Omega)} \le C \|\nabla f\|_{L^\pp(\Omega)}.
\end{equation}
Our example shows that this case, which was not considered in Theorem~\ref{theorem:sp-general-intro}, requires hypotheses that control the variation of $\pp$ inside of the domain, such as the classical log-H\"older condition. This example also shows that the Sobolev inequality~\eqref{eqn:sobolev-ineq},
\[\|f\|_\qq \le C \|\nabla f\|_{p(\cdot)}, \]
also fails when $\qq=\ps$.  As we noted above, this gives a positive answer to a question in~\cite[Problem~A.22]{CruzUribeFiorenza}.  
 
 Finally, We  adapt our test functions to prove that the boundary condition $\tauLH(\Omega)$ is also not sufficient  for the validity of  the Korn inequality, or for the solution of the divergence equation.

\subsection*{Constructing the exponent function}
Let $\Omega\subset \R^n$ be a bounded domain (not necessarily a John domain).  Fix sequences $\{c_k\}_{k=1}^\infty \subset \Omega$ and $\{r_k\}_{k=1}^\infty$ such that  $\{B(c_k,7r_k)\}_{k=1}^\infty$ is a collection of pairwise disjoint balls contained in $\Omega$. In each ball $B(c_k,7r_k)$ define  the balls $A_k^1=B(a_k,r_k)$, and the annuli $A_k^2=B(a_k,2r_k)\setminus B(a_k,r_k)$ and $A_k^3=B(a_k,3r_k)\setminus B(a_k,2r_k)$. Analogously, we define the ball $B_k^1=B(b_k,r_k)$ and the annuli $B_k^2$ and $B_k^3$ in such a way that $A_k^3\cap B_k^3=\emptyset$.  Let  $D_k$ be the complement in $B(c_k,7r_k)$ of the union of these balls and washers.   See Figure~\ref{interior test functions}.
\begin{figure}[ht]
  \centering
  \includegraphics[scale=0.3]{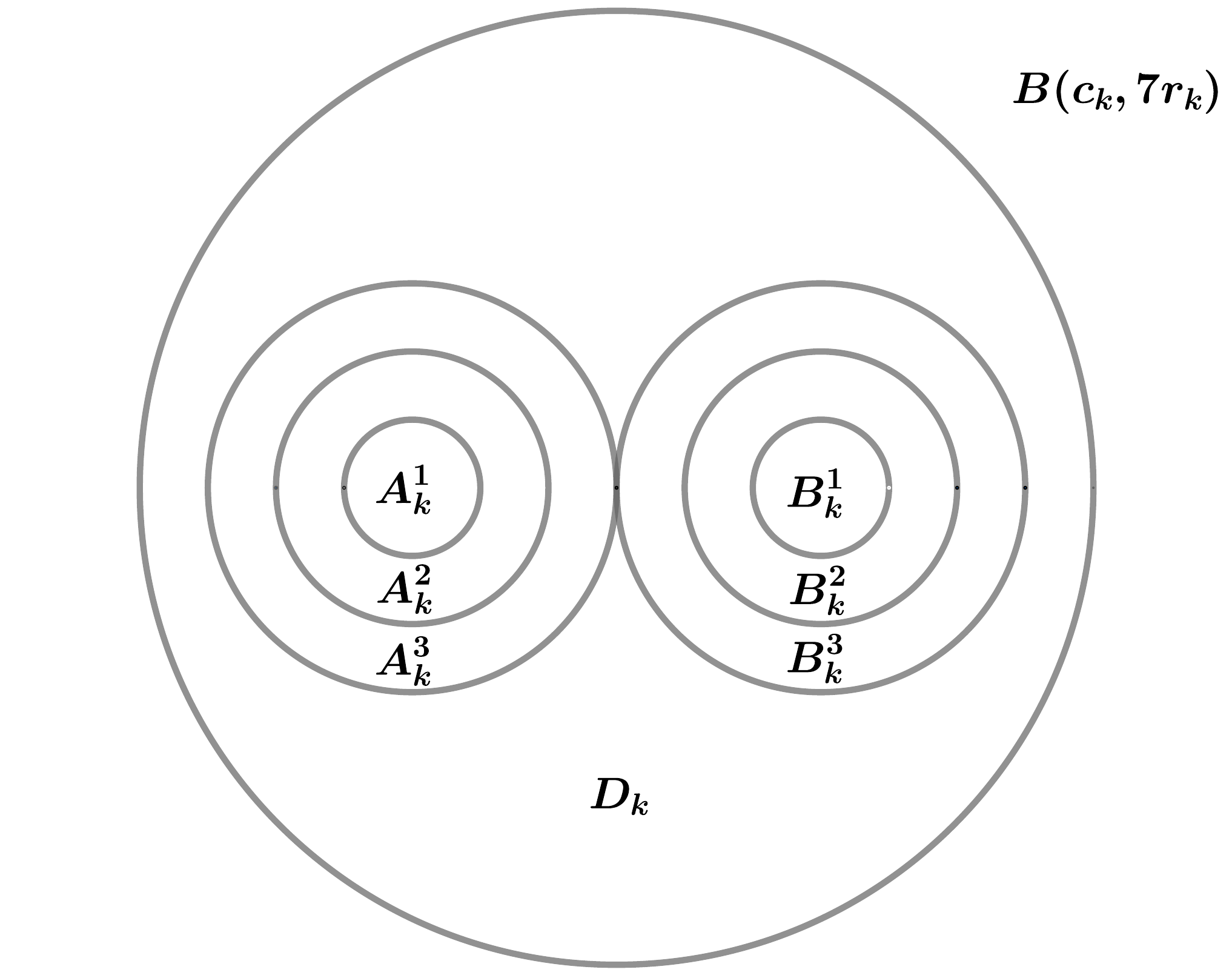}
  \caption{Local definition of the exponent $\pp$ and  test functions $f_k$}
  \label{interior test functions}
\end{figure}
Observe that since $\Omega$ has a finite measure and the balls $B(c_k,7r_k)$ are pairwise disjoint, then $\lim_{k\to\infty} r_k=0$. In addition, we also assume that $0<r_k<1$.

Now, in order to define the variable exponent $\pp \in \Pp(\Omega)$, let $p_0>1$ and let $\{p_k\}_{k=0}^\infty$ be a sequence that converges to $p_0$, with $p_k\geq p_0$ for all $k\in \N$. Hence, $\pp$ equals $p_0$ on $\Omega \setminus \bigcup_{k=1}^\infty B(c_k,7r_k)$, and on each of the balls $B(c_k,7r_k)$ is piece-wise constant or radial.  More precisely, let 
\begin{equation}\label{p counterexample}
\arraycolsep=1.4pt\def\arraystretch{1.5}
p(x) = \left\{\begin{array}{cl} 
    p_0 & \textrm{in } D_k\cup A_k^3 \cup B_k^3\\
    p_k & \textrm{in } A_k^1\cup B_k^1\\
    \left(\frac{|x-a_k|-r_k}{r_k}\right)p_0 + \left(\frac{2r_k-|x-a_k|}{r_k}\right)p_k  \hspace{1cm} & \textrm{in } A_n^2\\ 
    \left(\frac{|x-b_k|-r_k}{r_k}\right)p_0 + \left(\frac{2r_k-|x-b_k|}{r_k}\right)p_k  \hspace{1cm} & \textrm{in } B_k^2.
    \end{array}\right.
\end{equation}
Then $\pp \in \Pp(\Omega)$ and satisfies
\[1<p_0=p_-(\Omega)\leq p_+(\Omega)= \|p_k\|_{\ell^\infty} <\infty.\]
In addition, using that  $\{p_k\}_{k=0}^\infty$ converges to $p_0$, it follows by a straightforward estimate that $\pp$ is uniformly continuous on $\Omega$.  Fix $\varepsilon>0$; then there exists $k_0$ such that $|p_k-p_0|< \varepsilon$ for any $k> k_0$.  For $k\leq k_0$, since $\pp$ is uniformly continuous on each ball $B(c_k,7r_k)$, there exist $\delta_k$ such that $|p(x)-p(y)|<\varepsilon$ for any $x,\,y \in B(c_k,7r_k)$ such that $|x-y|<\delta_k$. Therefore, if we set
\[\delta:=\min\{\delta_1,\cdots,\delta_{k_0},r_1,\cdots,r_{k_0}\}, \]
then we have that  $|p(x)-p(y)|<\varepsilon$ for any $x,\,y \in \Omega$ such that $|x-y|<\delta$.

Finally, we choose a sequence $\{p_k\}$ such that the exponent $\pp$ satisfies all the assumptions made above but \eqref{Boundary LH constant} fails to be uniformly bounded along the balls $\{B(c_k,7r_k)\}_{k=1}^\infty$. We do so by showing the failure of the equivalent property in Lemma~\ref{lemma:LHBequiv}.  Define
\begin{equation*}
p_k:=p_0+\dfrac{1}{|\log(r_k)|^{1/2}}.
\end{equation*}
Hence, given $B=B(c_k,7r_k)$, we have 
\begin{multline}\label{non Holder}
|B|^{\frac{1}{p_+(B)}-\frac{1}{p_-(B)}} =|B|^{\frac{1}{p_k}-\frac{1}{p_0}} \simeq r_k^{\frac{n(p_0-p_k)}{p_0p_k}} \\
=\exp{\left(\frac{n}{p_0p_k}\frac{-1}{|\log(r_k)|^{1/2}}\log(r_k)\right)}
\geq C\exp{\left(\frac{n}{p_0\|p_k\|_{\ell^\infty}}|\log(r_k)|^{1/2}\right)},
\end{multline}
which tends to infinity.

\subsection*{Necessity of the $\tauLH(\Omega)$ condition for Theorem~\ref{theorem:sp-general-intro}}
We now show that the boundary log-H\"older condition $\tauLH(\Omega)$ is necessary for the improved Sobolev-Poincar\'e inequality \eqref{eq:sp-john} to hold. We start with a bounded domain $\Omega\subset\R^n$ and the uniformly continuous exponent $\pp\in \Pp(\Omega)$ defined above in \eqref{p counterexample}.  Further, we  also assume that the collection of balls $\{B(c_k,7r_k)\}_{k=1}^\infty$ approaches the boundary.  More precisely, we assume that  $d(c_k,\partial\Omega)=7r_k$. Then, from \eqref{non Holder} and Lemma~\ref{lemma:LHBequiv}, we have that $\pp$ does not satisfy the $\tauLH(\Omega)$ for $\tau=1$. Since these conditions are nested, it does not satisfy it for any  $\tau\geq 1$.  

We now define our  test functions to correspond to this geometry.  For each $k\in \N$, define  $f_k:\Omega\to \R$ so that each function $f_k$ has support in $B(c_k,7r_k) $ and  has mean value zero:
\begin{equation}\label{test f}
\arraycolsep=1.4pt\def\arraystretch{1.2}
f_k(x) = \left\{\begin{array}{ll} 
    r_k & \textrm{in } A_k^1 \cup A_k^2\\
    3r_k-|x-a_k|  \hspace{1cm} & \textrm{in } A_k^3\\ 
    -r_k & \textrm{in } B_k^1 \cup B_k^2\\
    |x-b_k|-3r_k  \hspace{1cm} & \textrm{in } B_k^3\\ 
    0 & \textrm{in } D_k.
    \end{array}\right.
\end{equation}
Clearly, each $f_k$ is a bounded Lipschitz function, and so $f_k \in W^{1,\pp}(\Omega)$.  
Now, fix $\alpha$ as in~\eqref{eq:alpha-intro} and define $\qq$  by ~\eqref{eq:def q-intro}. Let $q_k$ be such that $1/q_k = 1/p_k-\alpha/n$. 
Now suppose  that the improved Sobolev-Poincar\'e inequality is valid on $W^{1,p(\cdot)}(\Omega)$ for the exponents $\qq$ and $\pp$. But, if we let $B=B(c_k,7r_k)$, then  we have that
\begin{multline*}
\dfrac{\|f_k\|_{L^\qq(\Omega)}}{\|d^{1-\alpha}\grad f_k\|_{\Lp(\Omega)}}
= \dfrac{\|f_k\|_{L^\qq(\Omega)}}{\|d^{1-\alpha}\grad f_k\|_{\Lp(A_k^3\cup B_k^3)}}
\geq  \dfrac{\|f_k\|_{L^\qq(A_k^1)}}{\|d^{1-\alpha}\grad f_k\|_{\Lp(A_k^3\cup B_k^3)}} \\ 
 \simeq \dfrac{\|f_k\|_{L^\qq(A_k^1)}}{r_k^{1-\alpha}\|\grad f_k\|_{\Lp(A_k^3\cup B_k^3)}} 
 =\dfrac{r_k |A_k^1|^{\frac{1}{q_k}}}{r_k^{1-\alpha}(2|A_k^3|)^{\frac{1}{p_0}}}
 \simeq r_k^{-\alpha+\frac{n}{q_k}-\frac{n}{p_0}}
= r_k^{\frac{n}{p_k}-\frac{n}{p_0}} 
\simeq |B|^{\frac{1}{p_+(B)}-\frac{1}{p_-(B)}}.
 \end{multline*}
By \eqref{non Holder} we have that the right-hand term tends to infinity as $k\rightarrow \infty$.  This contradicts our assumption that the Sobolev-Poincar\'e inequality holds.   This shows that boundary condition $\tauLH(\Omega)$ is (in some sense) a necessary condition for  inequality~\eqref{eq:sp-john} to hold. 

\begin{comment}
If $B\cap B(c_k,7r_k)\setminus D_k \neq \emptyset$, for some one or more $k\in\N$, we fix $k$ so that $p_k$ is the largest value among the balls it intersects. Notice that $B$ might intersect $B(c_k,7r_k)\setminus D_k$ for infinitely many $k$. However, since the collection of exponents $p_k$ converges to $p_0$ from above there is a maximum value for this intersection. Also, the radius $r$ of $B$ is at least $\frac{r_k}{4}$, and 
%
\[ p_-(B(c_k,7r_k)) \leq p_-(B) \leq p_+(B) \leq p_+(B(c_k,7r_k)).\]
%
It follows that for any ball $B=B(x,d(x,\partial\Omega))$,
%
\[ |B|^{\frac{1}{p_+(B)}-\frac{1}{p_-(B)}} \leq C. \]
%
Therefore, by Lemma~\ref{lemma:LHBequiv}, $\pp \in \partial LH_0^1(\Omega)$.
\end{comment}

\subsection*{Insufficiency of the $\tauLH(\Omega)$ condition for the Sobolev-Poincar\'e inequality}
We now modify the previous construction to show that the $\tauLH(\Omega)$ condition is not sufficient to prove~\eqref{eqn:ps-endpt}.  
This example shows that it is also necessary to impose a control on the regularity of $\pp$ inside  the domain.  For this example, we will modify $\pp$ and the $f_k$ by assuming that the balls $\{B(c_k,7r_k)\}_{k=1}^\infty$ do not approach the boundary.  More precisely, we assume that the distance from each ball $B(c_k,7r_k)$ to $\partial \Omega$ is larger than some $\mu>0$, independent of $k$.  It then follows immediately that  $\pp$ satisfies the boundary condition $\tauLH(\Omega)$, for any $\tau\geq 1$.  It does not, however, satisfy the $LH_0(\Omega)$ condition, since this condition is equivalent to the quantity $|B|^{\frac{1}{p_+(B)}-\frac{1}{p_-(B)}}$ being uniformly bounded for all balls $B\subset \Omega$.  (See~\cite[Lemma~3.24]{CruzUribeFiorenza}; this result is stated for $\Omega=\R^n$ but the same proof works in general.)

We can now argue as before.  On each ball $B=B(c_k,7r_k)$ we have that
\begin{multline*}
\dfrac{\|f_k\|_{\Lps(\Omega)}}{\|\grad f_k\|_{\Lp(\Omega)}}
= \dfrac{\|f_k\|_{\Lps(\Omega)}}{\|\grad f_k\|_{\Lp(A_k^3\cup B_k^3)}}
\geq  \dfrac{\|f_k\|_{\Lps(A_k^1)}}{\|\grad f_k\|_{\Lp(A_k^3\cup B_k^3)}}\\
=\dfrac{r_k |A_k^1|^{\frac{1}{p^*_k}}}{(2|A_k^3|)^{\frac{1}{p_0}}}
\simeq r_k^{\frac{n}{p_k}-\frac{n}{p_0}} 
\simeq |B|^{\frac{1}{p_+(B)}-\frac{1}{p_-(B)}}.
\end{multline*}
Again, by~\eqref{non Holder} the righthand term is unbounded as $k\rightarrow \infty$.

\subsection*{Insufficiency of the $\tauLH(\Omega)$ condition for other inequalities}
It is well-known that in the constant exponent case 
the Sobolev-Poincar\'e inequality  is related to (and,  in fact, in many case  equivalent to) the Korn inequality, the conformal Korn inequality, the Fefferman-Stein inequality on bounded domains, and the solvability of the divergence equation. These inequalities have been studied in the variable exponent setting assuming the classical log-H\"older condition. Here we show that the $\tauLH(\Omega)$ condition is not sufficient for the Korn inequality to hold. 

We first recall this inequality. Given a bounded domain $\Omega \subset \R^n$, with $n \geq 2$ and $1 < p < \infty$, the constant exponent Korn inequality says that there exists a constant $C$ such that   
\begin{equation}\label{Korn}
    \| \grad\mathbf{u} \|_{L^p(\Omega)^{n \times n}} \leq C \| \varepsilon(\mathbf{u}) \|_{L^p(\Omega)^{n \times n}}
\end{equation}
for any vector field $\mathbf{u}$ in  $W^{1,p}(\Omega)^n$ with $\int_\Omega \grad\mathbf{u}-\varepsilon(\mathbf{u})=0$. By $\nabla\mathbf{u}$ we denote the differential matrix of $\mathbf{u}$ and by $\varepsilon(\mathbf{u})$ its symmetric part,
\begin{equation*}
\varepsilon_{ij}(\mathbf{u}) = \frac{1}{2} \left( \frac{\partial u_i}{\partial x_j} + \frac{\partial u_j}{\partial x_i} \right).
\end{equation*}
This inequality plays a fundamental role in the analysis of the linear elasticity equations, where $\mathbf{u}$ represents a displacement field of an elastic body. We refer to \cite{AD} for a detailed description.  The Korn inequality is also valid on spaces with variable exponent  assuming that $\pp$ verifies log-H\"older condition on the domain. (See~\cite[Theorem 14.3.23]{DieningBook} for an equivalent version of \eqref{Korn}.)  

We now construct our counter-example.  Define $\pp$ as in the previous example so that it satisfies the log-H\"older condition only on the boundary.  On each ball $B=B(c_k, 7r_k)$, define the vector field $\mathbf{u}:\Omega\to\R^n$  by 
\begin{equation}\label{test u}
\arraycolsep=1.4pt\def\arraystretch{1.2}
\mathbf{u}(x) = \left\{\begin{array}{ll} 
    S\cdot (x-a_k) & \textrm{in } A_k^1 \cup A_k^2,\\
    \phi(x)S\cdot (x-a_k)  \hspace{1cm} & \textrm{in } A_k^3,\\ 
    -S\cdot (x-b_k) & \textrm{in } B_k^1 \cup B_k^2,\\
    -\phi(x)S\cdot (x-b_k)  \hspace{1cm} & \textrm{in } B_k^3,\\ 
    0 & \textrm{in } D_k,
    \end{array}\right.
\end{equation}
where $\phi(x)=3-|x|/r_k$ and $S\in \R^{n\times n}$ is the skew-symmetric matrix 
that equals $S_{12}=-1$, $S_{21}=1$, and zero in the rest of the entries. The vector field $\mathbf{u}$ depends on $k$ but we omit it in the notation for simplicity. Then we have that
\begin{multline*}
\dfrac{\|\grad\mathbf{u}\|_{\Lp(\Omega)}}{\|\varepsilon(\mathbf{u})\|_{\Lp(\Omega)}}
= \dfrac{\|\grad\mathbf{u}\|_{\Lp(\Omega)}}{\|\varepsilon(\mathbf{u})\|_{\Lp(A_k^3\cup B_k^3)}}
\geq  \dfrac{\|\grad\mathbf{u}\|_{\Lp(A_k^1)}}{\|\varepsilon(\mathbf{u})\|_{\Lp(A_k^3\cup B_k^3)}} \\ 
\simeq \dfrac{|A_k^1|^{\frac{1}{p_k}}}{(2|A_k^3|)^{\frac{1}{p_0}}}
\simeq r_k^{\frac{n}{p_k}-\frac{n}{p_0}} 
\simeq |B|^{\frac{1}{p_+(B)}-\frac{1}{p_-(B)}},
\end{multline*}
and again the right-hand side is unbounded.  

\medskip

\subsection*{Insufficiency of the $\tauLH(\Omega)$ condition for the solvability of the divergence equation}
Another problem related to the Korn and Sobolev-Poincar\'e inequalities is the solvability of the divergence equation. In the constant exponent case, the existence of a solution for this differential equation has been widely studied by a number of authors under different geometric assumptions on the domain. Given a bounded domain and an exponent $1<q<\infty$, we say that the divergence equation is solvable if there exists a constant $C$ such that, for any function $f\in L^q(\Omega)$ with vanishing mean value, there is a solution ${\bf v}\in W^{1,q}_0(\Omega)^n$ of the divergence equation $\div({\bf v})=f$ with the $L^q(\Omega)$ regularity estimate
\[\|\grad{ \bf v}\|_{L^q(\Omega)}\leq C \|f\|_{L^q(\Omega)}.\]
Bogovskii~\cite{Bogovskii} constructed an explicit representation of the solution $\bf{v}$ on star-shaped domains with respect to a ball using singular integral operators. Later, this representation was generalized to the class of John domains  in \cite{ADM}. Then, using the theory of Calderón-Zygmund operators and the boundedness of the Hardy-Littlewood maximal operator (e.g., assuming that $\pp\in LH_0(\Omega)$), this explicit solution on John domains was generalized to variable Lebesgue spaces in~\cite[Theorem~14.3.15]{DieningBook}.

It is also well-known that the solvability of the divergence equation implies the Korn inequality under very general assumptions. In our case, it only uses the norm equivalence \eqref{eqn:assoc-norm}, and so this argument extends to the variable exponent setting. Therefore, if $\pp$ is the exponent above for which the Korn inequality fails, the divergence equation is not solvable on $\Lp(\Omega)$. We refer to \cite[Theorem~14.3.18]{DieningBook} and \cite[Theorem~14.3.23]{DieningBook} for the implication to the Korn inequality.

\section{Log-H\"older continuity on the boundary}
\label{section:lh-boundary}

In this section we show that the $\tauLH(\Omega)$ condition implies that the exponent function can be extended to a log-H\"older continuous function on $\partial \Omega$.  We cannot prove this for all John domains:  we need to impose some regularity on the boundary.  More precisely, we will assume the following.

\begin{definition}\label{definition:boundary-john}
A bounded domain $\Omega$ in $\R^n$ is a John domain up to the boundary (or, simply, a boundary John domain) with parameter $\lambda>1$ if there exists a point $x_0\in\Omega$ such that, given any $y\in\overline{\Omega}$, there exists a rectifiable curve parameterized by arc length $\gamma : [0,\ell] \rightarrow \Omega$, with $\gamma(0)=y$, $\gamma(\ell)=x_0$, and $\lambda\, \dist(\gamma(t),\partial \Omega) \geq t$.
\end{definition}

This property holds, for example, if $\Omega$ has a Lipschitz boundary, but it also holds for a much larger class of domains:  for instance, the so-called semi-uniform domains.  For a precise definition and the relationship of these domains to John domains, see Aikawa and Hirata~\cite{MR2410379}.  It would be interesting to give an explicit example of a John domain that is not a boundary John domain.

\begin{proposition} \label{prop:boundary-john-LH0}
    Given a boundary John domain $\Omega$ with John parameter $\lambda>1$, let $\pp\in \Pp(\Omega)$ be such that $1\leq p_-\leq p_+<\infty$ and $\pp \in \tauLH(\Omega)$ for some $\tau>2\lambda$.  Then $\pp$ can be uniquely extended to a function on $\overline{\Omega}$ such that $\pp$ is log-H\"older continuous on $\partial \Omega$.  More precisely, given $x,\,y \in \partial \Omega$, $|x-y|<\frac{1}{2}$,
    \begin{equation} \label{eqn:bjlh-1}
|p(x)-p(y)| \leq \frac{C_0}{-\log(|x-y|)}.   
\end{equation}
\end{proposition}

\begin{proof}
Fix a point $x\in \partial \Omega=\overline{\Omega}\setminus \Omega$.  Let $\{x_k\}$ be a sequence of points in $\Omega$ such that $x_k \rightarrow x$ as $k\rightarrow \infty$ and $|x_k-x| \leq \lambda d(x_k)$.  Since $\Omega$ is a boundary John domain, such a sequence of points always exists:  choose the points $x_k$ to lie on the curve $\gamma$ connecting $x_0$ to $x$.  (Note that the length of the curve connecting $x$ to $x_k$ is always longer than $|x-x_k|$.)  We will refer to such sequences at nontangential approach sequences.

We claim that the sequence $\{p(x_k)\}$ is Cauchy.  To see this, fix $j,\,k \in \N$; without loss of generality we may assume that $d(x_j)\leq d(x_k)$.  Then
\[ |x_j-x_k| \leq |x_j-x| + |x_k-x| \leq \lambda\big( d(x_j)+ d(x_k) \big) \leq 2\lambda d(x_k).  \]
It is immediate $x_j \in B_{x_k,2\lambda} \subset B_{x_k,\tau}$. Hence, by the $\tauLH(\Omega)$ condition,
\[ |p(x_j)-p(x_k)| \leq p_+(B_{x_k,\tau}) - p_-(B_{x_k,\tau}) \leq \frac{C_0}{-\log(\tau d(x_k))}.
\]
Thus, as $j,\,k \rightarrow \infty$, $|p(x_j)-p(x_k)|\rightarrow 0$.  

Since $\{p(x_k)\}$ is Cauchy, it converges; denote this limit by $p(x)$.  This limit is unique in the sense that given any other nontangential approach sequence $\{y_k\}$ in $\Omega$ converging to $x$, the same argument shows that
\[ |p(x_k)-p(y_k)| \leq C_0\max\bigg\{ \frac{1}{-\log(\tau d(x_k))},\frac{1}{-\log(\tau d(y_k))} \bigg\}, 
\]
and the right-hand side tends to $0$ as $k\rightarrow \infty$.  This defines our extension of $\pp$ to $\overline{\Omega}$; note that by our definition we have $p_-(\overline{\Omega})=p_-(\Omega)$ and $p_+(\overline{\Omega})=p_+(\Omega)$.

\medskip

We will now prove that $\pp$ is log-H\"older continuous on $\partial \Omega$.  It will suffice to prove that \eqref{eqn:bjlh-1} holds for  $x,\,y\in \partial \Omega$ such that 
\[ |x-y| < \frac{\tau-2\lambda}{4\tau} < \frac{1}{2}.
\]
Since $p_+<\infty$, it always holds for $|x-y|\geq \frac{\tau-2\lambda}{4\tau}$ for a sufficiently large constant $C_0$.  

Fix such $x,\,y \in \partial\Omega$. Let ${x_k}$ be a nontangential approach sequence converging to $x$.  Since we may choose the points $x_k$ to lie on the curve connecting $x_0$ to $x$, we may assume that there exists $k_0\geq 1$ such that 
\[  |x-y| = d(x_{k_0}) \left(\frac{\tau - 2\lambda}{2}\right). 
\]
Further, by passing to a subsequence (that includes $x_{k_0}$) we may assume that $d(x_k)$ decreases to $0$. Let $\{y_k\}$ be a nontangential approach sequence converging to $y$. By passing to a subsequence, we may assume that for all $k\geq k_0$, 
\[ |y_k-y| \leq d(x_{k_0}) \left(\frac{\tau - 2\lambda}{2}\right).
\]

We can now estimate as follows:  for all $k\geq k_0$,
\[ |x_k - x_{k_0}| \leq |x_k - x| + |x_{k_0}-x| \leq \lambda d(x_k) + \lambda d(x_{k_0}) < \tau d(x_{k_0}). \]
Similarly,
\[ |y_k - x_{k_0}| \leq |y_k - y|+ |y-x|+|x-x_{k_0}| < 2d(x_{k_0})\left(\frac{\tau - 2\lambda}{2}\right) + \lambda d(x_{k_0})
\le \tau d(x_{k_0}). 
\]
Hence, $x_k,\,y_k \in B_{x_{k_0},\tau}$, and so by the $\tauLH(\Omega)$ condition and our choice of $k_0$,
\begin{multline*} 
|p(x_k) - p(y_k)| \leq p_+(B_{x_{k_0},\tau}) - p_-(B_{x_{k_0},\tau}) \\
\leq \frac{C_0}{-\log(\tau d(x_{k_0}))}
=  \frac{C_0}{-\log(\frac{2\tau}{\tau-2\lambda}|x-y|)}
\leq \frac{D_0}{-\log(|x-y|)};
\end{multline*}
the last inequality holds since the function $x\mapsto -\log(x)^{-1}$ is convex and $\frac{2\tau}{\tau-2\lambda}>1$.  Since this is true for all $k\geq k_0$, if we pass to the limit we get \eqref{eqn:bjlh-1}.  This completes the proof.
\end{proof}

\section{An application to PDEs}
\label{section:applications}

In this section we give an application of our results to elliptic PDEs.  In~\cite{MR4332462}, the first author, Penrod and Rodney studied a Neumann-type problem for a degenerate $\pp$-Laplacian.   The basic operator is the
$\pp$-Laplacian:  given an
exponent function $\pp$, let
\[ \Delta_\pp u = -\div ( |\nabla u|^{\pp-2} \nabla u ). \]
This operator arises in the calculus of variations as an example of
nonstandard growth conditions, and has been studied by a
number of authors: see~\cite{MR3308513, MR2639204,
  MR2291779, MR3379920} and the extensive references they contain.  In~\cite{MR4332462} they considered the degenerate version of this operator,
\[ Lu = - \div(|\sqrt{Q}\nabla u|^{\pp-2} Q\nabla u), \]
where $Q$ is a $n\times n$, positive semi-definite, self-adjoint, measurable
matrix function.   These operators have also been studied, though
nowhere nearly as extensively:  see, for instance, \cite{MR3585054,  MR2670139,
 MR3974098}.   

Here we consider a particular version of their result.
Let $\Omega$ be a bounded, open domain in $\R^n$, and let $Q$ be  defined on an open neighborhood of $\overline{\Omega}$. They showed that if 
$1<p_-\leq p_+<\infty$ and $|\sqrt{Q(\cdot)}|_{op} \in L^\infty(\Omega)$, then the existence of the Poincar\'e inequality
\[ \|f-f_\Omega\|_{L^\pp(\Omega)} \leq C\|\grad f\|_{L^\pp(\Omega)}
\]
is equivalent to the existence of a weak solution to
\begin{equation}\label{nprob}
\begin{cases}
\div\Big(\Big|\sqrt{Q}\nabla u\Big|^{\pp-2} Q\nabla u\Big) 
& = |f|^{\pp-2} \text{ \; in }\Omega\\
{\bf n}^T \cdot Q \nabla u &= 0\text{ \; on }\partial \Omega,
\end{cases}
\end{equation}
where ${\bf n}$ is the outward unit normal vector of
$\partial \Omega$.
Further, they showed that solutions must satisfy the $\Lp(\Omega)$ regularity condition,
\begin{equation}\label{hypoest}
\| u\|_{\Lp(\Omega)} \leq C_1 \|f\|_{\Lp(\Omega)}^{\frac{r_*-1}{p_*-1}},
\end{equation}
where $p_*$ and $r_*$ are defined by
 \begin{equation*} 
   p_* = \begin{cases}
     p_+, & \text{ if } \|\grad u\|_{\Lp(\Omega)} <1, \\
     p_-, & \text{ if }\|\grad u\|_{\Lp(\Omega)} \geq 1,
   \end{cases}
   \quad \text{ and } \quad
   r_* = \begin{cases}
     p_+, & \text{ if } \|f\|_{\Lp(\Omega)} \geq 1,\\
     p_-, & \text{ if } \|f\|_{\Lp(\Omega)} < 1.
   \end{cases}
      \end{equation*}

      If we combine this with our Sobolev-Poincar\'e inequalities, in particular with Corollary~\ref{cor:imp-poincare} and Theorem~\ref{thm:sp-big-pp}, we immediately get the following result.

      \begin{theorem}
          Let $\Omega$ be a bounded John domain.  Suppose $\pp \in \Pp(\Omega)$ is such that $\pp \in \partial LH_0^{\tau_K}(\Omega)$ and either
          \begin{enumerate}
              \item $1<p_-\leq p_+ < \infty$ and $\frac{1}{\pp}$ is $\frac{\sigma}{n}$-continuous for any $\sigma<1$, or

              \item $n\leq p_-\leq p_+ < \infty$ (and we make no assumptions on the interior regularity of $\pp$).
          \end{enumerate}
          Then the Neumann-type problem~\eqref{nprob} has a weak solution in $\Omega$ that satisfies the regularity condition~\eqref{hypoest}.
      \end{theorem}

      \begin{remark}
      Because the matrix $Q$ is allowed to be degenerate, the definition of a weak solution to~\eqref{nprob} is somewhat technical, though it reduces to the classical definition when $Q$ is the identity matrix (i.e., the operator is the $\pp$-Laplacian).  Also, this definition allows for domains $\Omega$ whose boundaries are rough and on which the the normal is not well-defined.  We refer the reader to~\cite{MR4332462} for more information.
      \end{remark}
      
\bibliographystyle{plain}
\bibliography{References}

\end{document}